\documentclass[pdflatex,sn-mathphys-num]{sn-jnl}

\usepackage{graphicx}%
\usepackage{multirow}%
\usepackage{amsmath,amssymb,amsfonts}%
\usepackage{amsthm}%
\usepackage{mathrsfs}%
\usepackage[title]{appendix}%
\usepackage{xcolor}%
\usepackage{textcomp}%
\usepackage{manyfoot}%
\usepackage{booktabs}%
\usepackage{algorithm}%
\usepackage{algorithmicx}%
\usepackage{algpseudocode}%
\usepackage{listings}%
\usepackage{lipsum}%
\usepackage{bm}%
\usepackage{subcaption}
\usepackage[final]{changes} 
\definechangesauthor[name={Renato}, color=red!70!black]{R}
\theoremstyle{thmstyleone}%
\newtheorem{theorem}{Theorem}
\newtheorem{problem}{Problem}%

\theoremstyle{thmstyletwo}%

\makeatletter
\newtheoremstyle{thmstylethree}%
  {10pt}
  {3pt}
  {\normalfont}
  {}
  {\bfseries}
  {.}
  { }
  {}
\makeatother

\theoremstyle{thmstylethree}%
\newtheorem{remark}{Remark}%
\newtheorem{assumption}{Assumption}%

\begin{document}

\title[Article Title]{Worst-Case Probability Bounds for Finite-Horizon Safety under Moment Uncertainty}


\author*[1]{\fnm{Renato} \sur{Loureiro}}\email{renato.loureiro@ifr.uni-stuttgart.de}

\author[1]{\fnm{Torbj\o rn} \sur{Cunis}}\email{torbjoern.cunis@ifr.uni-stuttgart.de}


\affil[1]{\orgdiv{Institute of Flight Mechanics and Controls}, \orgname{University of Stuttgart}, \orgaddress{\street{Pfaffenwaldring 27}, \city{Stuttgart}, \postcode{70569}, \state{Baden-Württemberg}, \country{Germany}}}




\abstract{This paper addresses the problem of estimating upper bounds on the probability that a dynamical system will enter an undesirable region at some point within a finite time horizon. The primary source of uncertainty lies in the system's initial state, for which only a finite set of moments is known or within a prescribed interval. To tackle this problem, we formulate a measure-based program and propose its relaxation using the moment-sum-of-squares (moment-SOS) framework. The corresponding dual problem is introduced as a functional program, which is subsequently strengthened into a sum-of-squares (SOS) program. Notably, this dual formulation bears a structural resemblance to classical barrier function techniques for certifying system safety, with the key distinction that it yields a probabilistic certificate. 
The effectiveness of the proposed approach is demonstrated through multiple case studies, including a case involving an object in-orbit.
}


\maketitle

\clearpage
\section{Introduction}\label{sec1}

In the design and operation of safety-critical systems, robustness and safety assurance are of paramount importance.  
A fundamental challenge is safeguarding the system against entering undesirable or critical regions of the state space. 
This problem is inherently difficult, as it involves chance constraints that couple probability estimation with trajectory optimization. The present paper focuses on the former aspect: estimating probabilities given a partial characterization of the system's initial state uncertainty. 

A common approach to probability estimation are Monte Carlo-like approaches~\cite{tempo_randomized_2005}, typically evaluated at a finite set of future time points. However, the results depend heavily on the assumed sampling distribution, commonly chosen to be a normal distribution with known mean and covariance. Such an approach does not provide a robust safety criterion, as the true initial distribution may deviate from the assumed one. Ideally, one would consider the worst-case initial distribution consistent with the known characteristics, thereby obtaining a conservative and yet a safe probability bound.

Notably, in the field of space safety awareness, the issue of uncertainty propagation is of utmost relevance~\cite{luo_review_2017}. Satellite operators need to accurately assess the risk and apply maneuvers to decrease the risk of a potential collision. 
Typical approaches use very simplistic models and very restrictive assumptions, such as very short time windows~\cite{akella_probability_2000}. To perform less aggressive maneuvers, considering a long time horizon in the risk assessment due to the system uncertainties is a must. But due to the inherent nonlinearity of the system, simple approaches such as linearization and Guassian assumption do not accurately characterize the system even if the initial system uncertainty distribution is fully described by a known Gaussian distribution. This calls for a more sophisticated methodology that enables the use of the nonlinear model without a very coarse approximation \replaced[id=R]{and accommodates non-Gaussian distributions of uncertainty and its propagation. }{and consider not only Gaussian distributions to characterize the uncertainty, or its propagation.}

This paper presents a methodology for estimating the risk of a dynamical system to become unsafe: here defined as the worst-case probability that the system reaches a specified unsafe set over a finite-time horizon due to the initial state uncertainty. Specifically, we aim to determine provable (tight) upper bounds on the worst-case probability that a system, starting from an initial state \replaced[id=R]{that follows a partially known distribution }{drawn from a partially known distribution}, enters the unsafe set within the considered time horizon. 

\clearpage
\subsection{Literature review}

The study of uncertainty and its assessment in dynamical systems is a long-established field that remains, to this day, an area of active research.
Some work has been done specifically on stochastic dynamical system, where the goal is to determine first exit time probabilities~\cite{kushner_stochastic_1967}. These rely on stochastic Lyapunov function to determine some upper bounds on the exit time probability. In a more recent work, an approach to attain upper bounds on the probability that the system enters some unsafe set is done via a barrier function formulation~\cite{santoyo_barrier_2019}. On the same line, but with a slightly different goal, in~\cite{miller_unsafe_2026} the authors formulate a linear program over measures to attain the worst-case probability under stochastic dynamics; although only tested on very small academic examples. Moreover, in~\cite{jasour_semidefinite_2015,jasour_convex_2017} the authors work intensively on tractable formulations by means of convex programs to handle chance constraints, e.g., these could represent some probability of being unsafe.

A related topic is peak estimation, which consists in estimating bounds on quantities of interest along the trajectories of a known dynamic system starting from a set~\cite{miller_peak_2021}, e.g., these quantities can be the maximum deviation in some direction. 

Similar studies are present in the literature, such as the exit time probability~\cite{miller_unsafe_2026}, where the intent is to estimate for a fixed time window and for an initial state, how likely it is for the system characterized by a stochastic differential equation to leave the safe set.

Another interesting approach to uncertainty propagation is presented in~\cite{fantuzzi_bounds_2016}, where the authors develop convex programs to obtain upper and lower bounds on functions of interest along the trajectories starting from an initial set, e.g., finding the maximum deviation of the mean position of a system at a certain time.

Similarly, but to estimate the region of attraction of a polynomial dynamical system, in~\cite{henrion_convex_2013-1} the authors rely on measures and their propagation through the dynamics to infer the region of attraction, through means of a convex program.
In~\cite{shen_probabilistic_2025}, the authors also tackle risk assessment, but tailored to discrete-time systems, where the goal is to obtain probabilistic reachable sets. This method is sample-based, and tackles stochastic dynamical systems.

The content of this paper is similar to~\cite{miller_unsafe_2026}, where the worst case probability is also estimated by means of a moment relaxation, specifically for stochastic dynamical systems governed by Itô diffusion processes~\cite{oksendal_stochastic_1998}. There, uncertainty enters through the injection of disturbances into the system along the time-window. Although the authors note that uncertainty in the initial measure can be incorporated into the optimization framework, they do not elaborate on the resulting formulation. In contrast, the present paper places this very issue at center stage: we systematically investigate how the underlying convex programs are structured when the initial state is subject to uncertainty, thereby offering a detailed and unified treatment that complements and extends the earlier analysis.
We go one step further and study the case where the moments are not fully known. Specifically, we assume certain bounds on the moments, and wish to get some tight upper bound on the worst case probability under these assumptions.

\vspace{0.3cm}
\noindent\textbf{Contributions.}
This work concerns \replaced[id=R]{}{with} the treatment of system uncertainty and how to quantify safety of the system over a finite-time horizon by means of estimation of a (potentially) tight upper bound on the worst-case probability. In summary: 
\begin{enumerate}
	\item We formulate the problem of determining the worst-case probability of unsafe behavior, i.e., entering an undesired region, under partial knowledge of the initial system uncertainty, as a measure program~(Sect.~\ref{sec:measure_program}), and propose a tractable relaxation to a moment program~(Sect.~\ref{sec:moment_program}).
	\item We derive the dual program of the measure program, leading to a functional program, and establish a connection to barrier-based safety certification for deterministic systems~(Sect.~\ref{sec:functional_program})\textemdash Bridging measure-theoretic programming with classical barrier methods, typically written as a functional program. Later, a strengthening of the attained functional program is shown by using sum-of-squares polynomials and the generalized $\mathcal{S}$-procedure~(Sect.~\ref{sec:functional_program}).
	\item We extend the approach to ambiguity sets, where only bounds on moments of the unknown distribution are known~(Sect.~\ref{sec:ambiguity}).
\end{enumerate}

It is important to clarify that this paper does not aim to develop new theoretical advances in uncertainty propagation \textit{per se}; all the optimization problems we present could be easily reconstructed from the existing work in the literature. Our primary contribution lies in tailoring these formulations to a specific class of dynamical systems; namely, those in which uncertainty, or the sole source of randomness, originates exclusively from the initial state. In doing so, we hope to provide a self-contained and accessible reference that compiles, in a clear and concise manner, the concrete programs required to address the risk assessment problem stated above.

\vspace{0.3cm}
\noindent\textbf{Outline.}
The paper is organized as follows. The problem of interest is defined in Sect.~\ref{sec:problem_setup}. Its linear reformulation with occupation measures is described in Sect.~\ref{sec:core}: the concept of occupation measure is formally defined in Sect.~\ref{sec:occu_measure}; the connection between system dynamics and measures is established in Sect.~\ref{sec:liouville_eq} via the Liouville equation; the weak formulation of the problem in terms of nonnegative measures is presented in Sect.~\ref{sec:measure_program}, alongside its dual form in Sect.~\ref{sec:functional_program}, i.e., the functional program. In Sect.~\ref{sec:moment_program}, the relaxation of the measure program is presented alongside the stricter version of the functional program by relying on moment-sum-of-squares hierarchy and sum-of-squares programming, respectively. Sect.~\ref{sec:ambiguity} enhances the base program by considering ambiguity sets on the noncentral moments. Application of the algorithms and numerical examples are described in Sect.~\ref{sec:application}. Concluding remarks are gathered in Sect.~\ref{sec:conclusion}.

\clearpage
\noindent\textbf{Notation.} Let us recall some standard notation and some specific notation related to measure and functional programming. 
Let $\mathbb{N}$ and $\mathbb{R}$ denote the set of nonnegative integers and real numbers, respectively. For $n\in\mathbb{N}$, let $\mathbb{N}^n$ and $\mathbb{R}^n$ denote the $n$-dimensional nonnegative integer tuple and the Euclidean space, respectively. For $k\in \mathbb{N}$, we refer to $\mathbb{N}^n_k$ and $[k]$ as the elements $\bm{\alpha} \in \mathbb{N}^n$ with $|\bm{\alpha}| = \sum_{i=1}^n \alpha_i \leq k$ and the set $\{1,\ldots,k\}$, respectively. 
We denote by $\mathcal{C}(\bm{K})$ the set of all real-valued continuous functions $f:\bm{K}\to\mathbb{R}$, and by $\mathcal{C}_+(\bm{K})$ its nonnegative subcone. Moreover, we denote by $\mathcal{C}^k(\bm{K})$ the set of all $k$-times continuously differentiable functions on $\bm{K}$.
The set of nonnegative Borel measures supported in $\bm{K}$ is $\mathcal{M}_+(\bm{K})$ and the vector space of signed Borel measures supported in $\bm{K}$ is $\mathcal{M}(\bm{K})=\mathcal{M}_+(\bm{K})-\mathcal{M}_+(\bm{K})$.

The sets $\mathcal{C}_+(\bm{K})$ and $\mathcal{M}_+(\bm{K})$ are in topological duality when $\bm{K}$ is compact, and they admit a duality pairing $\langle \cdot, \cdot \rangle$ by the Lebesgue integration: 
\begin{align*}
	\forall f \in \mathcal{C}_+(\bm{K}), \, \mu \in \mathcal{M}_+(\bm{K}): \quad \langle f, \mu \rangle = \int_{\bm{K}} f(\mathbf{s})\, \mathrm{d}\mu(\mathbf{s}). 
\end{align*}

On the space of signed measures $\mathcal{M}(\bm{K})$, we define the total variation norm by 
\begin{equation*}
	\|\mu\|_{\mathrm{TV}} = \inf \{ \langle 1, \mu^+ \rangle + \langle 1, \mu^- \rangle : \mu^+, \mu^- \in \mathcal{M}_+(\bm{K}), \,\, \mu^+ - \mu^- = \mu\}.
\end{equation*}
For nonnegative measure $\mu \in \mathcal{M}_+(\bm{K})$, this reduces to $\| \mu \|_{\mathrm{TV}} = \langle 1, \mu \rangle$, which we call the mass of $\mu$. A nonnegative measure with mass $1$, i.e., $\langle 1, \mu \rangle = 1$, is said to be a probability measure. 
The product measure between $\mu \in \mathcal{M}_+(\bm{K}_1)$ and $\nu \in \mathcal{M}_+(\bm{K}_2)$ is the unique measure $\mu \otimes \nu$ satisfying $\forall \bm{A}_1  \subseteq \bm{K}_1, \bm{A}_2\subseteq \bm{K}_2: (\mu \otimes \nu) (\bm{A}_1 \times \bm{A}_2) = \mu(\bm{A}_1) \nu (\bm{A}_2)$. Take $\mu,\nu \in \mathcal{M}_+(\bm{K})$, we say that $\nu$ is dominated by $\mu$, denoted by $\nu \leq \mu$, if $\nu(\bm{A}) \leq \mu(\bm{A})$ for all subsets $\bm{A} \subset \bm{K}$. Equivalently, it means that there exists a unique $\hat{\nu}\in\mathcal{M}_+(\bm{K})$ such that $\forall \bm{A}\subset \bm{K}: \nu(\bm{A}) + \hat{\nu}(\bm{A}) = \mu(\bm{A})$, equivalently, in shorthand notation, $\nu + \hat{\nu}=\mu$.

Let $\mathbb{R}[\mathbf{x}]$ be the ring of polynomials in the variables $\mathbf{x}=(x_1, \ldots, x_n)$ and let $\mathbb{R}_d[\mathbf{x}]$ be the  vector space of polynomials of degree at most $d$. A monomial $x_1^{\alpha_1}\cdots x_n^{\alpha_n}$ is expressed in the multi-index notation as $\mathbf{x}^{\bm{\alpha}}$.
A polynomial $p\in \mathbb{R}_d[\mathbf{x}]$ is written as 
\begin{align*}
	\mathbf{x} \mapsto p(\mathbf{x}) = \sum\nolimits_{\bm{\alpha}\in\mathbb{N}^n} p_{\bm{\alpha}} \mathbf{x}^{\bm{\alpha}}
\end{align*}
for some finite-dimensional vector of coefficients $\mathbf{p}=(p_{\bm{\alpha}})_{\bm{\alpha} \in \mathbb{N}_d^n}$. We denote the degree of a polynomial $p$ as $\operatorname{deg} p = \max \{|\bm{\alpha}| : p_{\bm{\alpha}} \ne 0 \}$.
Denote by $\Sigma[\mathbf{x}]\subset \mathbb{R}[\mathbf{x}]$ the subset of real-valued polynomials that are \textit{sum-of-squares} (SOS) polynomials, i.e., 
\begin{align*}
	\Sigma[\mathbf{x}] = \{ p \in \mathbb{R}[\mathbf{x}] \, : \, p(\mathbf{x}) = \sum\nolimits_{i=1}^m f_i^2(\mathbf{x}), \, f_1, \ldots, f_m \in \mathbb{R}[\mathbf{x}], \, m\in\mathbb{N} \}.
\end{align*}
Similar to $\mathbb{R}_d[\mathbf{x}]$, we denote $\Sigma_d[\mathbf{x}]$ as the set of SOS polynomials of degree of at most $d$.

\section{Problem setup}%
\label{sec:problem_setup}

We consider the continuous-time dynamical system described by the ordinary differential equation (ODE)
\begin{align}%
	\label{eq:system_original}
	\begin{aligned}
		\dot{\mathbf{x}}(t) = f(\mathbf{x}(t),t), \quad t\in \mathbb{R}_{\geq 0},
	\end{aligned}
\end{align}
where the state $\mathbf{x}(t)$ takes values in a compact set $\bm{X}\subset \mathbb{R}^n$. The vector field $f:\bm{X} \times \mathbb{R}_{\geq 0} \to \mathbb{R}^n$ is assumed to be smooth.
\replaced{Let $\bm{X}_0\subset \bm{X}$ be the set of initial states.}{}

For a given initial state $\mathbf{x}_0 \in \bm{X}_0$ \added{and $T\in \mathbb{R}_{\geq 0}$}, let 
\begin{align*}
	\psi(\cdot \, | \, \mathbf{x}_0): [0,T] \to \mathbb{R}^n
\end{align*}
denote the unique solution to~\eqref{eq:system_original} satisfying $\psi(0 \, |\, \mathbf{x}_0)=\mathbf{x}_0$. Such a solution is called an admissible trajectory if it is defined on the entire interval $\bm{T}=[0,T]$ and remains within $\bm{X}$ for all $t\in \bm{T}$.  The set of all admissible trajectories emanating from the initial set $\bm{X}_0$ is then defined as the union
\begin{align*}
	\bm{\Psi}(\bm{X}_0, T) = \bigcup_{\mathbf{x}_0 \in \bm{X}_0} \{ \psi(\cdot\, | \, \mathbf{x}_0) \}.
\end{align*}

This collection is endowed with a weighting induced by the initial probability measure $\mu_0$ supported on $\bm{X}_0$, which captures the uncertainty in the initial state.

The available information on the initial state is encoded through a finite set of known noncentral moments of $\mu_0$. For a multi-index $\bm{\alpha}=(\alpha_1, \ldots, \alpha_n) \in \mathbb{N}^n$, the $\bm{\alpha}$-noncentral moment of the measure $\mu_0$ is given by
\begin{equation*}	
	b_{\bm{\alpha}} = \langle \mathbf{x}^{\bm{\alpha}}, \mu \rangle = \langle x_1^{\alpha_1} \cdots x_n^{\alpha_n}, \mu \rangle.
\end{equation*}

We assume that for a finite index set $\bm{A}$, the moments $b_{\bm{\alpha}}$ are known exactly. Moreover, since $\mu_0$ is a probability measure, its zeroth-order moment satisfies 
\begin{align*}
	\langle 1, \mu_0 \rangle = \mu_0(\bm{X}_0) = 1.
\end{align*}

Given this setup, we aim to compute the worst-case probability that the system state enters an unsafe set $\bm{K}\subset \bm{X}$ at some time $t\in \bm{T}$, under all probability measures $\mu_0$ compatible with the prescribed moment information. This leads to the following optimization problem.

\begin{problem}%
	\label{prob:main}
	The principal problem of interest is
\begin{subequations}
	\label{eq:problem}
	\begin{align}
		\mathfrak{P}^\star \, = \sup_{\substack{\mu_0 \in \mathcal{M}_+(\bm{X}_0), \\ \, \tau \in \bm{T}}} \, \, & \int_{\bm{K}} \psi(\tau \, | \, \mathbf{x}_0)~\mathrm{d}\mu_0(\mathbf{x}_0)   \label{eq:prob}\\
			\mathrm{s.t.} \quad \, \, \, 
			& \langle \mu_0, \mathbf{x}^{\bm{\alpha}} \rangle  = b_{\bm{\alpha}} \quad \mathrm{for} \quad \bm{\alpha} \in \bm{A}, \label{subeq:known-moments} \\
			& \langle \mu_0, 1 \rangle = 1. \label{eq:mass_one}
	\end{align}
\end{subequations}
\end{problem}

Here, the integral in~\eqref{eq:prob} equals the probability that the state at time $\tau$ lies in $\bm{K}$.
The supremum is taken over all initial probability measures $\mu_0$ satisfying the moment constraints~\eqref{subeq:known-moments} and over all possible times $\tau \in \bm{T}$. The constraints~\eqref{subeq:known-moments} and~\eqref{eq:mass_one} ensure consistency with the available statistical information and the normalization of the probability measure, respectively.

Problem~\ref{prob:main} is an infinite-dimensional, nonconvex optimization problem over the space of probability measures. Its intrinsic difficulty arises from three compounding factors:
\begin{enumerate}
	\item the decision variable is a measure $\mu_0 \in \mathcal{M}_+(\bm{X}_0)$, rendering the feasible set infinite-dimensional;
	\item the flow map $\psi(\tau \, |\, \mathbf{x}_0)$ is a nonlinear function of $\mathbf{x}_0$, consequently, convexity in the space of moments is generally not preserved.
	\item the supremum over $\tau\in\bm{T}$ introduces a continuum of temporal constraints.
\end{enumerate}

To render Problem~\ref{prob:main} computationally tractable, we adopt the relaxation framework pioneered in~\cite{vinter_equivalence_1978}, which circumvents the intractable dependence on individual trajectories by reformulating the dynamics in terms of the evolution of probability measures. Specifically, instead of enforcing the ODE constraint~\eqref{eq:system_original} point-wise along each trajectory, we work with the Liouville equation (Sect.~\ref{sec:liouville_eq}) which governs the time evolution of the measure induced by the flow~\eqref{eq:system_original}. This lifts the problem to the space of measures and replaces the nonlinear flow map with a linear, albeit infinite-dimensional, partial differential constraint. 

To solve the infinite-dimensional linear programming problem numerically, we employ Lasserre's moment-sum-of-squares hierarchy~\cite{lasserre_global_2001}. By doing so, under some mild assumptions, a sequence of finite-dimensional semidefinite programs (SDPs) is generated whose optimal values converge monotonically to the optimum $\mathfrak{P}^\star$ of Problem~\ref{prob:main}. In this manner, the original problem is approximated by a hierarchy of convex SDPs that are amenable by off-the-shelf solvers, providing provable upper bounds on the worst-case unsafe probability at each relaxation level. A detailed explanation is provided in Sect.~\ref{sec:moment_program}.

\begin{remark}
	Problem~\ref{prob:main} can be easily extended to account for uncertainties in the moment information. This additional uncertainty is quite relevant since it allows even for the study of the sensitivity of the problem against small deviations in the supposed known information.
	Notably, this extension does not add any structural difficulty while treating the Problem~\ref{prob:main} or its relaxation in terms of occupation measures. A detailed treatment is presented in detail in Sect.~\ref{sec:ambiguity}
\end{remark}

\clearpage
\section{Preliminaries}%
\label{sec:core}

In this section, we introduce the concept of occupation measures and establish its connection to the trajectories of the dynamical system~\eqref{eq:system_original}. This formalism provides a natural framework for propagating initial uncertainty through the dynamics.

\subsection{Occupation measure}%
\label{sec:occu_measure}

Let $\mu_0 \in \mathcal{M}_+(\bm{X}_0)$ be a given initial probability measure, which we denote as the \textit{initial measure}. For a fixed initial condition $\mathbf{x}_0$, we define the \textit{conditional occupation measure} $\mu_{\mathrm{c}}(\cdot \, | \, \mathbf{x}_0) \in \mathcal{M}_+(\bm{T} \times \bm{X})$ as a measure that satisfies the equality
\begin{align*}
	\mu_{\mathrm{c}}(\bm{A}\times \bm{B}\, | \, \mathbf{x}_0) = \int_{\bm{T}} \mathbf{1}_{\bm{A}\times \bm{B}}(t,\psi(t \, |\, \mathbf{x}_0)) \, \mathrm{d}t,
\end{align*}
for all Borel sets $\bm{A}\subseteq \bm{T}$ and $\bm{B}\subseteq \bm{X}$, where $\mathbf{1}_{\bm{A}\times \bm{B}}$ denotes the indicator function of $\bm{A}\times \bm{B}$. Averaging over the initial distribution $\mu_0$, we obtain the~\textit{occupation measure} $\mu \in \mathcal{M}_+(\bm{T} \times \bm{X})$ defined by 
\begin{align*}
	\mu(\bm{A}\times \bm{B}) = \int_{\bm{X}_0} \mu_{\mathrm{c}}(\bm{A} \times \bm{B} \, | \, \mathbf{x}_0) \, \mathrm{d}\mu_0(\mathbf{x}_0).
\end{align*}
Thus, $\mu(\bm{A}\times \bm{B})$ quantifies the $\mu_0$-weighted total time that trajectories of the dynamical system~\eqref{eq:system_original} spend in the set $\bm{B}$ during the time interval $\bm{A}$. 

Additionally, we define the \textit{terminal measure} $\mu_{\mathrm{T}} \in \mathcal{M}_+(\bm{X})$, which characterizes the distribution of the state at the final time $T$ under the flow generated by~\eqref{eq:system_original}, as
\begin{align*}
	\mu_{\mathrm{T}}(\bm{B}) = \int_{\bm{X}_0} \mathbf{1}_{\bm{B}}(\psi(T \, | \, \mathbf{x}_0))\mathrm{d}\mu_0(\mathbf{x}_0) \quad \text{for all} \quad \bm{B} \subseteq \bm{X}.
\end{align*}

In summary, the occupation measure $\mu$ captures the time-averaged behavior of the flow defined by~\eqref{eq:system_original} with respect to~$\mu_0$, while the terminal measure $\mu_{\mathrm{T}}$ encodes the pushforward of the initial measure under the flow map~\eqref{eq:system_original} at time $T$. These measures serve as the fundamental objects in the measure-relaxation framework that follows; for a comprehensive treatment, see, e.g.,~\cite{miller_peak_2021}.

\subsection{Liouville equation}%
\label{sec:liouville_eq}

The occupation measures introduced in the preceding section are not independent \replaced{of each other}{}; they are linked through a fundamental conservation law that governs the transport of measures along the flow of the dynamical system~\eqref{eq:system_original}. This connection is captured by the Liouville equation, which we now outline.

Given~$\psi(\cdot \, | \, \mathbf{x}_0)$ denote the unique admissible trajectory of~\eqref{eq:system_original} emanating from $\mathbf{x}_0 \in \bm{X}_0$. For any test function $v\in \mathcal{C}^1(\bm{T} \times \bm{X})$, the chain rule along trajectories yields
\begin{equation}
\label{eq:chain_rule}
\begin{aligned}
	\frac{\mathrm{d}}{\mathrm{d}t} v(t, \psi(t \mid \mathbf{x}_0)) 
	&= \mathcal{L}_f v(t, \psi(t \mid \mathbf{x}_0)),
\end{aligned}
\end{equation}
where the linear operator $\mathcal{L}_f:\mathcal{C}^1(\bm{T}\times \bm{X}) \to \mathcal{C}(\bm{T} \times \bm{X})$ is defined as 
\begin{align}
\label{eq:lie_derivative}
\mathcal{L}_f: v \mapsto \frac{\partial v}{\partial t} + \sum_{i=1}^n \frac{\partial v}{\partial x_i} f_i.
\end{align}
This operator coincides with the Lie derivative of $v$ along the vector field \replaced[id=R]{$f$}{~\eqref{eq:system_original}}.

Integrating~\eqref{eq:chain_rule} over $t\in \bm{T}$ and averaging over the initial measure $\mu_0$, we obtain the weak (integral) form of the Liouville equation:
\begin{align}
\label{eq:liouv_weak}
\int_{\bm{X}} v(T, \mathbf{x})  \, \mathrm{d}\mu_{\mathrm{T}}(\mathbf{x}) - \int_{\bm{X}_0} v(0, \mathbf{x}) \, \mathrm{d}\mu_0(\mathbf{x})
= \int_{\bm{T} \times \bm{X}} \mathcal{L}_f v(t, \mathbf{x}) \, \mathrm{d}\mu(t,\mathbf{x}),
\end{align}
for every $v\in\mathcal{C}^1(\bm{T}\times \bm{X})$. In the more compact bracket notation, this reads
\begin{align}
	\label{eq:liouv}
	\langle v(T,\cdot), \mu_{\mathrm{T}} \rangle = \langle v(0, \cdot), \mu_0 \rangle + \langle \mathcal{L}_f v, \mu\rangle.
\end{align}
for every $v\in\mathcal{C}^1(\bm{T}\times \bm{X})$.
Equivalently, in distributional form, the Liouville equation can be written as
\begin{align}
	\label{eq:liuoville_eq_brief}
	\delta_T \otimes \mu_{\mathrm{T}} = \delta_{0} \otimes \mu_0 + \mathcal{L}_f^\dagger \mu	
\end{align}
where $\delta_\tau$ denotes the Dirac measure at a point $\tau$, $\mathcal{L}_f^\dagger$ denotes the adjoint operator of $\mathcal{L}_f$ with respect to the duality pairing $\langle \cdot, \cdot \rangle$, i.e., $\mathcal{L}_f^\dagger$ is the unique operator satisfying 
\begin{align*}
	\forall v \in \mathcal{C}^1(\bm{T} \times \bm{X}), \, \mu \in \mathcal{M}_+(\bm{T} \times \bm{X}):\, \langle \mathcal{L}_fv, \mu \rangle = \langle v, \mathcal{L}_f^\dagger \mu \rangle.
\end{align*}

The measures $(\mu_0, \mu_{\mathrm{T}}, \mu)\in \mathcal{M}_+(\bm{X}_0) \times  \mathcal{M}_+(\bm{X}) \times \mathcal{M}_+(\bm{T}\times \bm{X})$ are precisely those defined in Sect.~\ref{sec:occu_measure}, the initial measure, the terminal measure and the occupation measure, respectively. Equation~\eqref{eq:liouv} is classical in fluid mechanics and statistical physics, where it governs the transport of particle densities in a flow. It characterizes the entire family of admissible trajectories emanating from the initial distribution $\mu_0$. Fig.~\ref{fig:visualize-liouv} provides a schematic illustration of this evolution.

\begin{remark}[Free terminal time]
	\label{remark:free}
	The Liouville equation \eqref{eq:liouv} can be generalized to the case where the terminal time is not fixed a priori. If the terminal measure is defined on the augmented space $\mu_{\mathrm{T}} \in \mathcal{M}_+(\bm{T} \times \bm{X})$, then the equation becomes
	\begin{align}
		\label{eq:liouville_short_free}
		\mu_{\mathrm{T}} = \delta_0 \otimes \mu_0 + \mathcal{L}_f^\dagger \mu,
	\end{align}
	which remains linear in the measures. This flexibility is particularly useful when optimizing over the time $\tau \in \bm{T}$, as sought in Problem~\ref{prob:main}. For further details see, e.g.,~\cite{miller_peak_2021,covella_uncertainty_2023,henrion_convex_2013-1}.
\end{remark}

\begin{remark}
	It is relevant to highlight that the triplet $(\mu_0,\mu_{\mathrm{T}},\mu)$ that satisfies~\eqref{eq:liouville_short_free} is a relaxed \replaced[id=R]{}{occupation} measure~\cite{miller_peak_2024}, i.e., the set of triplets that fulfill the condition~\eqref{eq:liouville_short_free} may be larger than the set of \replaced[id=R]{}{occupations} measures that are generated from the dynamics.
\end{remark}

The Liouville equation provides a linear relationship between the initial measure $\mu_0$, the terminal measure $\mu_{\mathrm{T}}$, and occupation measure $\mu$, thereby circumventing the nonlinearity inherent in the individual trajectories of the system~\eqref{eq:system_original}. Equipped with the measure-theoretic tools presented in this section, we are now in a position to formulate the infinite-dimensional linear program over measures that will serve as the foundation for tackling Problem~\ref{prob:main}. This reformulation is presented in the following section.

\begin{figure}[h]
    \centering
	\begin{minipage}{0.45\textwidth}
    	\includegraphics[width=0.9\linewidth]{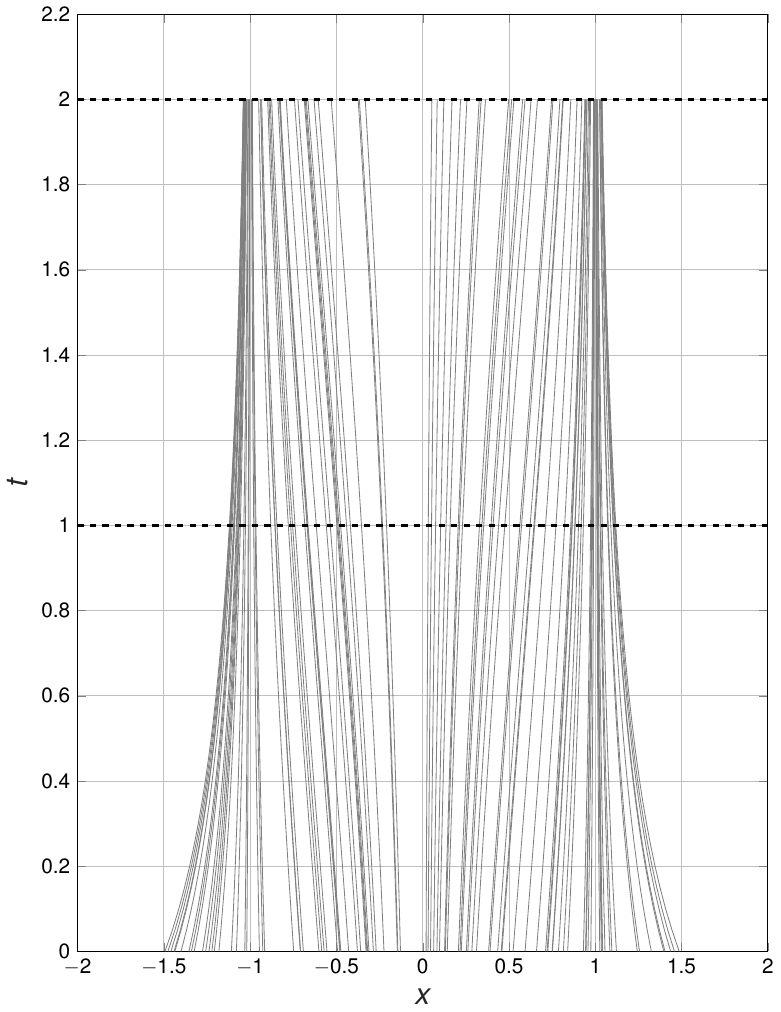}
	\end{minipage}%
	\begin{minipage}{0.01\textwidth}
		\hfill
	\end{minipage}%
	\begin{minipage}{0.48\textwidth}
		\includegraphics[width=0.9\linewidth]{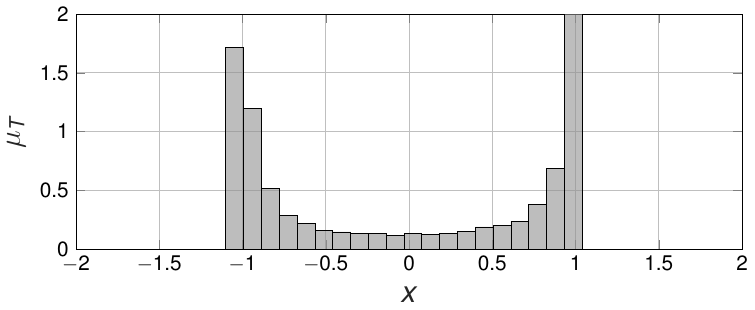}
		\includegraphics[width=0.9\linewidth]{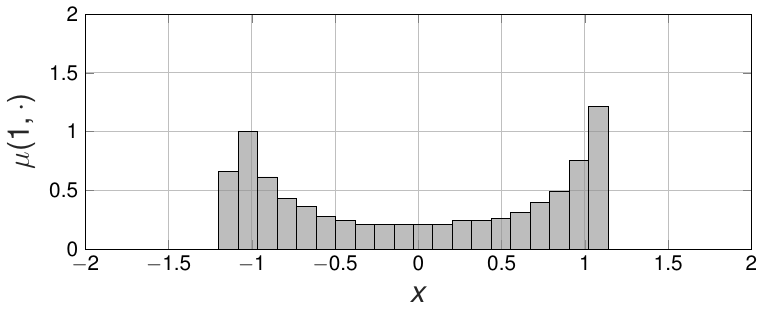}
		\includegraphics[width=0.9\linewidth]{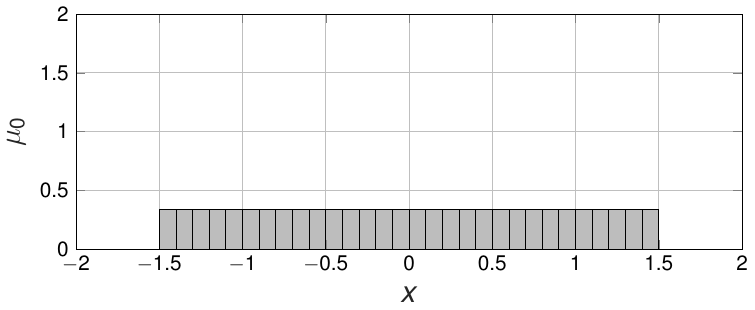}
	\end{minipage}
    \caption{Illustration of the evolution of the initial probability measure $\mu_0$, defined here as the uniform distribution on the interval $[-1.5, 1.5]$,  under the flow map $\dot{x}(t)=0.5x(1-x^2)$ up until time $T=1$.}%
    \label{fig:visualize-liouv}
\end{figure}

\section{Infinite-dimensional linear program}%
\label{sec:infinite_dimensional_program}

In this section, we formulate the probabilistic safety problem as an infinite-dimensional linear program over measures, and subsequently derive its dual functional counterpart.

\subsection{Measure program}%
\label{sec:measure_program}

The objective of this section is to construct a convex relaxation of Problem~\ref{prob:main} that yields a provable upper bound on the worst-case unsafe probability. To this end, we build upon the measure-theoretic framework developed in~\cite{henrion_approximate_2009}, which introduces an infinite-dimensional linear program over measures for approximating the indicator function of a set.

We begin by recalling a fundamental result from~\cite[Theorem 3.1]{henrion_approximate_2009}. Given a set $\bm{K}\subseteq \bm{X}$ and a measure $\mu_2\in\mathcal{M}_+(\bm{X})$, under some mild assumptions, the optimal value of the linear program 
\begin{align}
	\label{eq:chance-upper}
	\sup_{\mu_1} \, \{ \langle 1, \mu_1 \rangle \, : \, \mu_2 - \mu_1 \in \mathcal{M}_+(\bm{X}), \, \mu_1 \in \mathcal{M}_+(\bm{K})\}
\end{align}
\replaced{equates to }{}$\mu_2(\bm{K})=\langle \mathbf{1}_{\bm{K}}, \mu_2 \rangle $. In other words, the supremum of the mass of $\mu_1$ supported on $\bm{K}$, subject to the constraint that $\mu_2 -\mu_1$ remains a nonnegative measure on $\bm{X}$, exactly equates the mass that $\mu_2$ assigns to $\bm{K}$.

In our setting, we adapt this construction to the terminal measure $\mu_{\mathrm{T}}$ obtained from the Liouville equation~\eqref{eq:liouv}. Specifically, we introduce an auxiliary measure $\mu_{\mathrm{a}}\in\mathcal{M}_+(\bm{K})$ supported on the unsafe set $\bm{K}$, and we seek to maximize its total mass $\mu_{\mathrm{a}} (\bm{K})=\langle 1, \mu_{\mathrm{a}} \rangle $. 
By adding the constraint $\mu_{\mathrm{T}}-\mu_{\mathrm{a}} \in \mathcal{M}_+(\bm{X})$, we ensure that $\mu_{\mathrm{a}}$ is dominated by $\mu_{\mathrm{T}}$. 
Under some regularity assumptions, 
for a fixed final time $\tau \in  \bm{T}$, maximizing $\langle 1, \mu_{\mathrm{a}} \rangle$ over all admissible $(\mu_0, \mu_{\mathrm{T}}, \mu)$ yields the optimal value of Problem~\ref{prob:main}, i.e., $\mathfrak{P}^\star=\langle 1, \mu_{\mathrm{a}}^\star  \rangle $.
In words, the optimum value corresponds precisely to the worst-case probability that the state \replaced{at time $\tau$}{} lies in $\bm{K}$. This \replaced[id=R]{will be}{is} formally stated and proved in Theorem~\ref{theorem:no_gap}.

However, we are not interested in a fixed terminal time worst-case probability, we require the terminal time to be free within the interval $\bm{T}$. Hence, following Remark~\ref{remark:free}, we can easily do this by considering an augmented terminal measure $\mu_{\mathrm{T}} \in \mathcal{M}_+(\bm{T} \times \bm{X})$ and project it onto $\bm{X}$ via the marginalization operator~\cite{miller_safety_2023} introduced next.

Given $\eta \in \mathcal{M}_+(\bm{T}\times \bm{X})$, the projection-pushforward \replaced{operator}{} $\pi_\#^{\mathbf{x}}$ expresses the $\mathbf{x}$-marginal of $\eta$ with duality pairing
\begin{align*}
	\forall f \in \mathcal{C}(\bm{X}): \quad \langle f, \pi_\#^{\mathbf{x}} \eta \rangle = \int_{\bm{T}\times \bm{X}} f(\mathbf{x}) \, \mathrm{d}\eta(\mathbf{x},t).
\end{align*}
\replaced[id=R]{}{this is also called the $\mathbf{x}$-marginalization of a measure $\eta$.}
We are now in a position to state the central measure program. It reads \replaced[id=R]{}{as follows.}

\begin{subequations}%
	\label{eq:measure-program}
	\begin{align}
		P^\star=\sup_{\mu, \mu_0, \mu_{\mathrm{T}}, \mu_{\mathrm{a}}} \quad &\langle 1, \mu_\mathrm{a} \rangle \\
			\mathrm{s.t.} \qquad  & \mu_{\mathrm{T}} =  \delta_0 \otimes \mu_0 + \mathcal{L}^\dagger \mu \label{seq:liouv}\\
			& \langle \mu_0, 1 \rangle = 1 \label{seq:probability_1}\\
			& \langle \mu_0, \mathbf{x}^{\bm{\alpha}} \rangle = b_{\bm{\alpha}} \quad \mathrm{for} \quad \bm{\alpha} \in \bm{A} \label{seq:moment_cct}\\
			& \pi_\#^{\mathbf{x}} \mu_{\mathrm{T}} - \mu_{\mathrm{a}} \in \mathcal{M}_+(\bm{X}) \label{eq:domination_constraint} \\
			& \mu_0 \in \mathcal{M}_+(\bm{X}_0) \\
			& \mu_{\mathrm{a}} \in \mathcal{M}_+(\bm{K}) \label{seq:mu_a_K}\\
			& \mu, \mu_{\mathrm{T}} \in \mathcal{M}_+(\bm{T} \times \bm{X})
	\end{align}
\end{subequations}

Program~\eqref{eq:measure-program} is formulated for the free terminal time case, as reflected by the presence of the time variable in $\mu_{\mathrm{T}}$ and the \replaced[id=R]{projection-pushforward}{marginalization} operator $\pi_\#^\mathbf{x}$ in~\eqref{eq:domination_constraint}. The fixed terminal time case is recovered by taking $\mu_{\mathrm{T}}\in\mathcal{M}_+(\bm{X})$, i.e., independent of time, replacing $\pi_\#^\mathbf{x}\mu_{\mathrm{T}}$ with $\mu_{\mathrm{T}}$ directly, and using the Liouville equation~\eqref{eq:liuoville_eq_brief} for a fixed terminal time.

The proof follows the same lines as~\cite[Theorem 1]{miller_unsafe_2026}, but with three key modifications: we replace the stochastic dynamical system with a deterministic one, impose additional moment conditions on the initial measure, and handle the extra constraint associated with the introduced auxiliary variable $\mu_\mathrm{a}$. 

\begin{theorem}%
	\label{theorem:measure_program}
	The infinite-dimensional linear program over measures~\eqref{eq:measure-program} provides an upper bound on the optimal value $\mathfrak{P}^\star$ of Problem~\ref{prob:main}.
\end{theorem}
\begin{proof}
	 Let $(\mu_0, \tau)$ be any feasible \replaced[id=R]{solution to}{pair of}~\eqref{eq:problem}. Hence,~\eqref{seq:probability_1} and~\eqref{seq:moment_cct} are satisfied. Define the occupation measure $\mu \in \mathcal{M}_+(\bm{T} \times \bm{X})$ by
	\begin{equation*}
		\langle v, \mu \rangle = \int_{\bm{X}_0} \int_0^\tau v(t, \psi(t\mid \mathbf{x}_0)) \, \mathrm{d}t \, \mathrm{d}\mu_0(\mathbf{x}_0)
	\end{equation*}
	for all $v \in \mathcal{C}(\bm{T} \times \bm{X})$,
	and the terminal measure $\mu_{\mathrm{T}} \in \mathcal{M}_+(\bm{T} \times \bm{X})$ by
	\begin{equation*}
		\langle v, \mu_{\mathrm{T}} \rangle = \int_{\bm{X}_0} v(\tau, \psi(\tau \mid \mathbf{x}_0)) \, \mathrm{d}\mu_0(\mathbf{x}_0)
	\end{equation*}
	for all $v \in \mathcal{C}(\bm{T} \times \bm{X})$.
	For any test function $v \in \mathcal{C}^1(\bm{T} \times \bm{X})$, integration by parts gives
	$
		\langle v, \mu_{\mathrm{T}} \rangle - \langle v(0, \cdot), \mu_0 \rangle = \langle \mathcal{L}_fv, \mu \rangle,
	$
	which is precisely the Liouville constraint $\mu_{\mathrm{T}} = \delta_0 \otimes \mu_0 + \mathcal{L}_f^\dagger \mu$.
	Now define $\mu_{\mathrm{a}} \in \mathcal{M}_+(\bm{K})$ as the restriction of $\pi_\#^{\mathbf{x}} \mu_{\mathrm{T}}$ to $\bm{K}$,
	\begin{equation*}
		\langle 1, \mu_{\mathrm{a}} \rangle = \mu_{\mathrm{T}}(\bm{T} \times \bm{K}) = \int_{\bm{X}_0} \mathbf{1}_{\bm{K}}(\psi(\tau \mid \mathbf{x}_0)) \, \mathrm{d}\mu_0(\mathbf{x}_0) = \int_{\bm{K}} \psi(\tau \mid \mathbf{x}_0) \, \mathrm{d}\mu_0(\mathbf{x}_0),
	\end{equation*}
	leading to the satisfaction of both~\eqref{eq:domination_constraint} and~\eqref{seq:mu_a_K} by construction.
	Thus, every feasible point of~\eqref{eq:problem} maps to a feasible point of the measure relaxation $(\mu, \mu_0, \mu_\mathrm{a}, \mu_\mathrm{T})$ with the same objective value $\mathfrak{P}^\star=\langle 1, \mu_\mathrm{a}\rangle$. Taking the supremum yields $P^\star \geq \mathfrak{P}^\star$.
\end{proof}

A more interesting result than Theorem~\ref{theorem:measure_program} is in which conditions the optimal solution of~\eqref{eq:measure-program} equals that of Problem~\ref{prob:main}. To ensure that that occurs we make the following assumption.

\begin{assumption}%
	\label{ass:compactness}
	The sets $\bm{X}$, $\bm{X}_0\subseteq \bm{X}$, $\bm{K}\subseteq \bm{X}$, and $\bm{T}$ are compact. All trajectories starting from $\bm{X}_0$ remain in $\bm{X}$ for all $t\in\bm{T}$.
\end{assumption}

\begin{remark}
	Assumption~\ref{ass:compactness} is very subtle, namely, related to the choice of $\bm{X}$, since it requests that the chosen $\bm{X}$ is invariant with respect to any trajectory obtained from $\bm{X}_0$ during the interval $\bm{T}$. Prior knowledge of this set is extremely relevant, since if invariance does not hold, the underlying programs that assume this assumption will become unfeasible, mostly because of the conservation of mass condition via the Liouville equation. One can just use an arbitrary sufficiently large $\bm{X}$ based on empirical results (e.g., some randomly samples trajectories), or make use of formal reachability techniques~\cite{althoff_set_2021}, i.e., estimation of an outer approximation of the reachable set of the system starting from $\bm{X}_0$ during $\bm{T}$.
\end{remark}

\clearpage
First, we see that under Assumption~\ref{ass:compactness} the program~\eqref{eq:measure-program} is solvable; the proof follows the rationale taken in~\cite[Proof of Theorem 2.3]{lasserre_nonlinear_2007}. 

\begin{theorem}%
	\label{theorem:inf_min}
	If Assumption~\ref{ass:compactness} holds and~\eqref{eq:measure-program} is feasible, then~\eqref{eq:measure-program} is solvable, i.e., $\sup = \max$.
\end{theorem}%
\begin{proof} 
	For any $(\mu_{\mathrm{T}}, \mu,\mu_0)$ that fulfills~\eqref{seq:liouv}, i.e.,
	\begin{equation*}
		\forall v \in \mathcal{C}^1(\bm{T} \times \bm{X}):\quad \langle v, \mu_{\mathrm{T}} \rangle - \langle v(0, \cdot), \mu_0 \rangle = \langle \mathcal{L}_fv, \mu \rangle,
	\end{equation*}
	by considering the particular cases $v(t,\mathbf{x})=1$ and $v(t,\mathbf{x})=T-t$, it follows that
	\begin{equation*}
		\| \mu_0\|_{\mathrm{TV}} =1, \quad \|\mu_{\mathrm{T}} \|_{\mathrm{TV}} = 1, \quad \| \mu \|_{\mathrm{TV}} = \langle \mu_{\mathrm{T}}, t \rangle \leq T \|\mu_{\mathrm{T}}\|_{\mathrm{TV}} \leq T,
	\end{equation*}
	where $\|\nu\|_{\mathrm{TV}}$ denotes the total variation norm of the signed measure $\nu$, but in the particular case of nonnegative measures, this reduces to the mass of $\nu$, i.e., $\langle 1, \nu \rangle$.
	Now considering the domination constraint~\eqref{eq:domination_constraint}, taking the particular case $v(t,\mathbf{x})=1$, gives
	\begin{equation*}
		\|\mu_\mathrm{a}\|_{\mathrm{TV}} \leq \| \pi_\#^{\mathrm{x}} \mu_\mathrm{T}\|_{\mathrm{TV}} = \|\mu_\mathrm{T}\|_{\mathrm{TV}}= 1,
	\end{equation*}
	hence the feasible set of~\eqref{eq:measure-program} is bounded in total variation norm, and  
	$(\mu,\mu_\mathrm{T},\mu_\mathrm{a},\mu_0)$ belongs to some closed ball $\mathcal{B}_R$ in $\mathcal{M}^\sim=\mathcal{M}(\bm{T}\times \bm{X}) \times \mathcal{M}(\bm{T}\times \bm{X}) \times \mathcal{M}(\bm{K}) \times \mathcal{M}(\bm{X}_0)$. 
	Since $\bm{X}$, $\bm{X}_0$, $\bm{T}$, and $\bm{K}$ are compact by Assumption~\ref{ass:compactness}, and by virtue of the Banach-Alaoglu Theorem~\cite[Lemma 1.3.2 (b)]{hernandez-lerma_markov_2003}, $\mathcal{B}_R$ is compact for the weak-$\star$ topology, and additionally, due to the compactness assumption, the cone $\mathcal{M}^\sim_+$ is weak-$\star$ closed.
	Moreover, the flow~\eqref{eq:system_original} is smooth, thus $\mathcal{L}_f^\dagger$ is weak-$\star$ continuous~\cite[Remark 2.1 (i)]{lasserre_nonlinear_2007}. The projection-pushforward operator $\pi_\#^\mathbf{x}$ is also weak-$\star$ continuous. 
	The remaining equality constraints (mass and moments) are continuous, since they are defined by continuous linear functionals.
	Therefore, the feasible set is a closed subset of $\mathcal{B}_R \cap \mathcal{M}^\sim_+$, and thus is compact. 
	Since~\eqref{eq:measure-program} is feasible, by compactness, a maximizing sequence $(\mu^n, \mu_0^n, \mu_\mathrm{T}^n, \mu_\mathrm{a}^n)$ has a weak-$\star$ convergent subsequence to a feasible limit.  

	Finally, since $\bm{K}$ is compact due to Assumption~\ref{ass:compactness}, the linear functional $\mu_\mathrm{a} \mapsto \langle 1, \mu_\mathrm{a} \rangle$ to be maximized is weak-$\star$ continuous, and the limit attains the supremum. Hence, $\sup=\max$.
\end{proof}

Now we state the main result that establishes tightness between the original Problem~\ref{prob:main} and its relaxation via occupation measures.

\begin{theorem}%
	\label{theorem:no_gap}
	If Assumption~\ref{ass:compactness} holds and~\eqref{eq:measure-program} is feasible, then $P^\star=\mathfrak{P}^\star$.
\end{theorem}
\begin{proof}
	Since $f(\cdot, \cdot)$ is smooth over $\bm{T}\times \bm{X}$ as assumed in the system dynamics~\eqref{eq:system_original}, the state space is compact by Assumption~\ref{ass:compactness} and the cost function is weak-$\star$ continuous, then by similar arguments as in~\cite[Theorem 2.1]{lewis_relaxation_1980}, $P^\star=\mathfrak{P}^\star$.
\end{proof}

A natural dual counterpart to~\eqref{eq:measure-program} also exists, taking the form of a functional optimization problem over continuous functions~\cite{lasserre_global_2001}. This dual perspective offers complementary insights, revealing connections to classical results from \textit{dissipativity theory}, such as the synthesis of barrier functions for safety verification. The functional program and its relationship to the primal measure program are explored in detail below.

\subsection{Functional program}%
\label{sec:functional_program}

We now derive the dual of the measure program \eqref{eq:measure-program}. This dual formulation offers a complementary perspective: instead of optimizing over measures, it seeks functions that certify the upper bound on the unsafe probability, i.e., the worst-case probability. As we shall see, this functional program is \replaced[id=R]{}{intimately} related to the classical barrier function synthesis for safety verification~\cite{WIELAND2007462}, where the existence of a suitable barrier function guarantees that the system remains within a safe set.

Following the standard Lagrangian duality procedure for infinite-dimensional linear programs over measures (see, e.g.,~\cite{lasserre_global_2001}), we obtain \replaced[id=R]{}{the following dual program.}
  \begin{subequations}
	\label{eq:functional_program}
    \begin{align}
      D^\star \, = \inf_{\substack{\omega \in \mathcal{C}^1(\bm{X}\times \bm{T}), \\ \gamma \in \mathbb{R}^{|\bm{A}|}, \, \nu \in \mathbb{R}}} \quad & - \nu-\sum_{\bm{\alpha} \in \bm{A}} \gamma_{\bm{\alpha}} b_{\bm{\alpha}}  \label{eq:cost_functional}\\
      \mathrm{s.t.} \,\,\, \quad \quad & \forall (\mathbf{x},t) \in \bm{K} \times \bm{T}:  \hspace{-1cm} &  \hspace{-2cm} \omega(\mathbf{x},t) & \geq 1 \label{eq:on_unsafe}\\
      & \forall (\mathbf{x},t) \in \bm{X} \times \bm{T}:  \hspace{-1cm} & \hspace{-2cm} \omega(\mathbf{x},t) &\geq 0 \label{eq:nonnegativity_of_w}\\
      & \forall (\mathbf{x},t) \in \bm{X}\times \bm{T}:  \hspace{-1cm} & \hspace{-2cm} -\mathcal{L}_f \omega(\mathbf{x},t) &\geq 0  \label{eq:dissipativity_cond}\\
	  & \forall \mathbf{x} \in \bm{X}_0:  \hspace{-1cm} & \hspace{-2.5cm} -\omega(\mathbf{x},0) - \nu-\sum_{\bm{\alpha}\in \bm{A}} \gamma_{\bm{\alpha}} \mathbf{x}^{\bm{\alpha}} &\geq 0 \label{eq:moment_info_restriction}
    \end{align}
  \end{subequations}
  with $ b_{\bm{\alpha}} = \langle \mu_0, \mathbf{x}^{\bm{\alpha}} \rangle$ for $\bm{\alpha} \in \bm{A}$.
  The extensive step-by-step dualization of~\eqref{eq:measure-program} is shown in Appendix~\ref{sec:appendix_A}. 

The following result shows that under some regularity assumptions, the measure program and its dual, i.e., the functional program~\eqref{eq:functional_program}, have the same optimal value. Here we follow a similar proof to~\cite[Proof of Theorem 2.3 (ii)]{lasserre_nonlinear_2007} and~\cite[Proof of Theorem 3.1.2]{korda_moment-sum--squares_2016}.
	
\begin{theorem}%
	\label{theorem:functional}
	Under Assumption~\ref{ass:compactness}, the optimal value of~\eqref{eq:measure-program} coincides with the optimal value of~\eqref{eq:functional_program}, i.e., $P^\star = D^\star$.
\end{theorem}%
\begin{proof} 
	First let us write~\eqref{eq:measure-program} in the standard format of a linear program.
	We introduce the new measure decision variable $\mu_\mathrm{z} \in \mathcal{M}_+(\bm{X})$ to convert the inequality constraint~\eqref{eq:domination_constraint} into an equality constraint, making the new decision variable vector
	$
		\mathfrak{y} = (\mu_0, \mu_\mathrm{a}, \mu, \mu_\mathrm{T}, \mu_\mathrm{z}),
	$
	with the conic constraints 
	\begin{align*}
		\mathfrak{K} = \mathcal{M}_+(\bm{X}_0) \times \mathcal{M}_+(\bm{K})  \times  \mathcal{M}_+(\bm{T} \times \bm{X}) \times  \mathcal{M}_+(\bm{T} \times \bm{X}) \times \mathcal{M}_+(\bm{X}).
	\end{align*}
	The new set of equalities
	\begin{align*}
		\begin{aligned}
			& \mu_\mathrm{T} =  \delta_0 \otimes \mu_0 + \mathcal{L}_f^\dagger \mu \\
			& \langle \mu_0, 1 \rangle = 1 \\
			& \langle \mu_0, \mathbf{x}^{\bm{\alpha}} \rangle = b_{\bm{\alpha}} \quad \mathrm{for} \quad \bm{\alpha} \in \bm{A} \\
			& \pi_\#^{\mathbf{x}} \mu_\mathrm{T} - \mu_\mathrm{a} - \mu_\mathrm{z} = 0 
		\end{aligned}
	\end{align*}
	can be represented simply by $\mathfrak{A}\mathfrak{y} = \mathfrak{b}$, since $\mathcal{L}_f^\dagger$, $\pi_\#^\mathbf{x}$, and the duality pairing are linear operators.
	Thus,~\eqref{eq:measure-program} is equivalently written as the standard-form linear program
	\begin{align}%
		\label{eq:lp_standard}
		\begin{aligned}
			\sup_{\mathfrak{y}} \quad &\langle \mathfrak{y}, \mathfrak{c} \rangle \\
			\mathrm{s.t.} \quad  & \mathfrak{A}\mathfrak{y} = \mathfrak{b} \\
			& \mathfrak{y}\in \mathfrak{K}
		\end{aligned}
	\end{align}
	with $\langle \mathfrak{y}, \mathfrak{c} \rangle = \langle 1, \mu_\mathrm{a} \rangle$.
	Let us consider the set 
	$
		\mathfrak{D}=\{(\mathfrak{A}\mathfrak{y}, \langle \mathfrak{y}, \mathfrak{c} \rangle ): \mathfrak{y} \in \mathfrak{K}\}
	$. 
	By \cite[Theorem 3.10]{anderson_linear_1987},
	if the linear program~\eqref{eq:lp_standard} has a feasible solution with finite value, and if $\mathfrak{D}$ is closed in the weak-$\star$ topology, then there is no duality gap. 
	We now prove $\mathfrak{D}$ is closed. Consider the sequence $(\mathfrak{y}_n)_{n\in \mathbb{N}} \in \mathfrak{K}$, such that
	\begin{align*}
		(\mathfrak{A}\mathfrak{y}_n, \langle \mathfrak{c}, \mathfrak{y}_n \rangle ) \to (a,b)
	\end{align*}
	for some $(a,b)\in \mathfrak{S} \times \mathbb{R}$, with $\mathfrak{S}=\mathcal{M}(\bm{T}\times \bm{X}) \times \mathbb{R}\times \mathbb{R}^{|\bm{A}|} \times \mathcal{M}(\bm{X})$. In particular, similar to the proof of Theorem~\ref{theorem:inf_min}, by considering the cases $v_0=T-t$ and $v_1=1$, the obtain
	\begin{align*}
		\begin{aligned}
			\| \mu_0\|_{\mathrm{TV}} =1, \quad \|\mu_\mathrm{T} \|_{\mathrm{TV}} = 1, \quad \| \mu \|_{\mathrm{TV}} = \langle \mu_\mathrm{T}, t \rangle \leq T \|\mu_\mathrm{T}\|_{\mathrm{TV}} \leq T,
		\end{aligned}
	\end{align*}
	and due to the rewriting of the domination constraint~\eqref{eq:domination_constraint} with the introduction of $\mu_z$, we have
	\begin{align}
		\label{eq:help_proof}
		\|\mu_\mathrm{z}\|_{\mathrm{TV}} + \|\mu_\mathrm{a}\|_{\mathrm{TV}} = 1 \implies \|\mu_\mathrm{z}\|_{\mathrm{TV}} \leq  \|\mu_\mathrm{a}\|_{\mathrm{TV}} \leq 1.
	\end{align}

	Therefore, all components of $\mathfrak{y}_n$ are uniformly bounded in total variation norm. \replaced[id=R]{By}{Due to} Assumption~\ref{ass:compactness}, $\bm{X}, \bm{X}_0, \bm{K}$, and $\bm{T}$ are compact, and by the Banach-Alaoglu theorem~\cite[Lemma 1.3.2 (b)]{hernandez-lerma_markov_2003} the closed balls in the corresponding measure spaces are weak-$\star$ compact. Hence, there exists a subsequence $(\mathfrak{y}_{n})$ converging weak-$\star$ to some $\mathfrak{y} \in \mathfrak{K}$.
	By weak-$\star$ continuity of $\mathfrak{A}$ and the linear functional $\langle \mathfrak{c}, \cdot \rangle$ (see, e.g., \cite[Remark 2.1 (i)]{lasserre_nonlinear_2007}), we have $\mathfrak{A}\mathfrak{y}=a$ and $\langle \mathfrak{c}, \mathfrak{y} \rangle = b$. Thus, $(a,b) \in \mathfrak{D}$, proving $\mathfrak{D}$ is closed. The fact that the optimal value of~\eqref{eq:lp_standard} is finite follows readily from~\eqref{eq:help_proof} and from the compactness of $\bm{K}$.

	The result now follows trivially from~\cite[Theorem 3.10]{anderson_linear_1987}.
\end{proof}

\begin{remark}
The functional program~\eqref{eq:functional_program} admits a natural interpretation as synthesis of a  barrier function $\omega:\bm{X}\times\bm{T} \to \mathbb{R}$ for safety verification. Indeed, the constraints on $\omega$ resemble those of a time-dependent barrier function:
\begin{itemize}
	\item constraint~\eqref{eq:on_unsafe} requires $\omega$ to be at least $1$ on the unsafe set $\bm{T}\times\bm{K}$, i.e.,
	\begin{align*}
		\bm{T} \times \bm{K} \subseteq \{ \mathbf{x} \in \bm{X} \, : \, \omega(\mathbf{x}) \geq 1\};
	\end{align*}
	\item constraint~\eqref{eq:nonnegativity_of_w} enforces nonnegativity of $\omega$ on $\bm{T}\times \bm{X}$;
	\item constraint~\eqref{eq:dissipativity_cond} imposes a dissipativity-like condition, i.e., 
	\begin{align*}
		\mathcal{L}_f \omega(\mathbf{x},t) \leq 0 \,\,  \text{on} \,\,  \bm{X}\times \bm{T},
	\end{align*}
	which ensures that $\omega$ decreases along the trajectories within $\bm{X}\times \bm{T}$;
	\item constraint~\eqref{eq:moment_info_restriction} encodes the moment information through $\nu+\sum_{\bm{\alpha}\in\bm{A}} \gamma_{\bm{\alpha}} \mathbf{x}^{\bm{\alpha}}$ in the set containment constraint 
	\begin{align*}
		\bm{X}_0 \subseteq \left\{ \mathbf{x} \in \bm{X} \,:\, \omega(\mathbf{x}, 0) \geq -\nu -\sum_{\bm{\alpha} \in \bm{A}} \gamma_{\bm{\alpha}} \mathbf{x}^{\bm{\alpha}} \right\}.
	\end{align*}
\end{itemize}
This structure is characteristic of barrier function methods, where the existence of a function satisfying such inequalities, specifically~\eqref{eq:on_unsafe}\textendash\eqref{eq:dissipativity_cond}, certifies that trajectories originating from a specified initial set remain confined to the safe region. In the present probabilistic setting, adding~\eqref{eq:moment_info_restriction} provides an upper bound on the probability of reaching the unsafe set. The dual variables $\bm{\gamma}=(\gamma_{\bm{\alpha}})_{\bm{\alpha}\in\bm{A}}$ and $\nu$ serve to incorporate the moment information into the barrier certificate, effectively shaping the initial condition constraint to reflect the available statistical knowledge. Thus, the functional program can be viewed as a generalized barrier function synthesis problem, with the objective of minimizing an upper bound on the unsafe probability.
\end{remark}

\section{Finite-dimensional approximations}%
\label{sec:finite_dimensional}

The infinite dimensional problems~\eqref{eq:measure-program} and its dual~\eqref{eq:functional_program} are not directly amenable for computation. However, a sequence of finite dimensional approximations in terms of SDPs can be obtained. 
To do so, we first introduce the necessary background on moment sequences and their associated linear functionals and matrices, to develop a finite relaxed version of~\eqref{eq:measure-program}. That is, we briefly introduce the Moment-SOS hierarchy~\cite{lasserre_global_2001}. Later, when dealing with the strengthened version of the functional program~\eqref{eq:functional_program}, we recall SOS polynomials~\cite{blekherman_semidefinite_2013} and the generalized $\mathcal{S}$-procedure~\cite{Tan2006}.

We make the following assumptions.
\begin{assumption}%
	\label{ass:poly_dynamics}
	The dynamics presented in~\eqref{eq:system_original} are a polynomial, $f \in \mathbb{R}[\mathbf{x},t]^n$.
\end{assumption}
\begin{remark}
	The framework presented in this paper tailored for polynomial dynamical systems readily extends to rational dynamics provided the denominator is algebraic; this extension is detailed in Appendix~\ref{sec:appendix_C}. 
\end{remark}

\begin{assumption}%
	\label{ass:all_sets}
	The sets $\bm{T}$,  $\bm{X}$, and $\bm{X}_0$ are compact basic semialgebraic sets described as 
	\begin{align*}
		\begin{aligned}
			\bm{T}   &= \{t \in \mathbb{R} \, : \, g_i^{\bm{T}}(t) \geq 0, i \in [m_{\bm{T}}] \}, \\
			\bm{X}   &= \{\mathbf{x} \in \mathbb{R}^n \, : \, g_i^{\bm{X}}(\mathbf{x}) \geq 0, i \in [m_{\bm{X}}] \}, \\
			\bm{X}_0 &= \{\mathbf{x} \in \mathbb{R}^n \, : \, g_i^{\bm{X}_0}(\mathbf{x}) \geq 0, i \in [m_{\bm{X}_0}]  \}, \\
		\end{aligned}
	\end{align*}
	where $g_i^{\bm{T}} \in \mathbb{R}[t]$, and $g_i^{\bm{X}}(\mathbf{x}), g_i^{\bm{X}_0}(\mathbf{x}) \in \mathbb{R}[\mathbf{x}]$.
\end{assumption}
\begin{assumption}
	\label{ass:unsafe_set}
	The unsafe set $\bm{K}\subset \bm{X}$ is the compact basic semialgebraic set
	\begin{align}
		\bm{K} = \{\mathbf{x}\in \bm{X} \, : \, g_i^{\bm{K}}(\mathbf{x})\geq 0, \, i\in [m_{\bm{K}}] \},
	\end{align}
	where $g_i^{\bm{K}} \in \mathbb{R}[\mathbf{x}]$.
\end{assumption}


\subsection{Moment program}%
\label{sec:moment_program}

The concepts presented in this subsection are extensively developed in~\cite{lasserre_global_2001,lasserre_moments_2010,lasserre_introduction_2015}. 
Given a sequence $\mathbf{y}=(y_{\bm{\alpha}})_{\bm{\alpha} \in \mathbb{N}^n}$, we define the \textit{Riesz functional} $L_{\mathbf{y}}:\mathbb{R}[\mathbf{x}]\to \mathbb{R}$ by
\begin{align*}
	L_{\mathbf{y}}\left(\sum_{\bm{\alpha} \in \mathbb{N}^n} p_{\bm{\alpha}} \mathbf{x}^{\bm{\alpha}} \right) := \sum_{\bm{\alpha} \in \mathbb{N}^n} p_{\bm{\alpha}} y_{\bm{\alpha}},
\end{align*}
where only finitely many coefficients $p_{\bm{\alpha}} \in \mathbb{R}$ are nonzero.

A central result is the Riesz-Haviland theorem~\cite[Theorem 2.34]{lasserre_introduction_2015}, which provides a measure-theoretic characterization of when a given sequence $\mathbf{y}=(y_{\bm{\alpha}})_{\bm{\alpha}\in\mathbb{N}^n}$ admits a representing nonnegative measure supported on a closed set $\bm{K} \subset \mathbb{R}^n$. It states that given $\mathbf{y}=(y_{\bm{\alpha}})_{\bm{\alpha}\in\mathbb{N}^n}$ and a closed set $\bm{K}\subset \mathbb{R}^n$, there exists a finite Borel nonnegative measure $\mu$ on $\bm{K}$ with moments $\int_{\bm{K}} \mathbf{x}^{\bm{\alpha}} \mathrm{d}\mu(\mathbf{x}) = y_{\bm{\alpha}}$ if and only if $L_{\mathbf{y}}(f) \geq 0$ for all polynomials $f\geq 0$ on $\bm{K}$.
Since checking nonnegativity for all polynomials is generally intractable~\cite{lasserre_introduction_2015}, a common relaxation is to consider SOS polynomials. 

For each $d\in \mathbb{N}$, the \textit{moment matrix} of order $d$ associated with $\mathbf{y}=(y_{\bm{\alpha}})_{\bm{\alpha}\in\mathbb{N}^n}$ is the real symmetric matrix $\mathbf{M}_d(\mathbf{y})$ whose rows and columns are indexed in $\mathbb{N}_d^n$ such that 
\begin{align*}
	(\mathbf{M}_d(\mathbf{y}))_{(\bm{\alpha},\bm{\beta})} = {L}_{\mathbf{y}}(\mathbf{x}^{\bm{\alpha}+\bm{\beta}})=y_{\bm{\alpha}+\bm{\beta}}, \,\bm{\alpha}, \bm{\beta} \in \mathbb{N}_d^n.	
\end{align*}

Equivalently, for any polynomial $q = \sum_{\bm{\alpha}} q_{\bm{\alpha}} \mathbf{x}^{\bm{\alpha}} \in\mathbb{R}_d[\mathbf{x}]$ with coefficient vector $\mathbf{q}=(q_{\bm{\alpha}})_{|\bm{\alpha}|\leq d}$, the Riesz functional $L_{\mathbf{y}}$ satisfies
$ 
	L_{\mathbf{y}}(q^2) = \langle \mathbf{q}, \mathbf{M}_d(\mathbf{y}) \mathbf{q} \rangle \geq 0,
$
so $\mathbf{M}_d(\mathbf{y}) \succeq 0$ is a necessary condition for the existence of a nonnegative measure $\mu$ with pseudomoments $\mathbf{y}$.


Given a polynomial $g\in\mathbb{R}[\mathbf{x}]$ and a sequence $\mathbf{y}=(y_{\bm{\alpha}})_{\bm{\alpha}\in\mathbb{N}^n}$, the \textit{localizing matrix} of order $d$ associated with $g$ and $\mathbf{y}$ is the real symmetric matrix $\mathbf{M}_d(\mathbf{y} | g)$ whose rows and columns are indexed by $\mathbb{N}_d^n$ and whose entries are given by
\begin{align*}
	(\mathbf{M}_d(\mathbf{y} | g))_{(\bm{\alpha},\bm{\beta})}= L_{\mathbf{y}}(g(\mathbf{x})\mathbf{x}^{\bm{\alpha}+\bm{\beta}}),\, \bm{\alpha}, \bm{\beta} \in \mathbb{N}_d^n.
\end{align*}
Equivalently, for any polynomial $q\in\mathbb{R}_d[\mathbf{x}]$ with coefficient vector $\mathbf{q}=(q_{\bm{\alpha}})_{|\bm{\alpha}|\leq d}$, we have
$
	L_{\mathbf{y}}(gq^2) = \langle \mathbf{q}, \mathbf{M}_d(\mathbf{y} | g) \mathbf{q} \rangle
$.

Now suppose $\mu$ is a finite Borel measure supported on the semialgebraic set
\begin{align*}
	\bm{K}=\{\mathbf{x} \in \mathbb{R}^n : g_j(\mathbf{x}) \geq 0, \, j\in[m] \},
\end{align*}
and suppose $\mathbf{y}$ are the moments of $\mu$. Then for each $j\in[m]$ and any $q\in\mathbb{R}_d[\mathbf{x}]$, we obtain
$
	\langle \mathbf{q}, \mathbf{M}_d(\mathbf{y} | g_j) \mathbf{q} \rangle
	= L_{\mathbf{y}}(g_j q^2)
	\geq 0,
$
because $g_j\ge 0$ on $\bm{K}$ and $q^2\ge 0$. Hence, $\mathbf{M}_d(\mathbf{y} | g_j) \succeq 0$ for all $j\in[m]$ is a necessary condition for a nonnegative measure $\mu$ supported in $\bm{K}$ with pseudomoments $\mathbf{y}$. 

Together with the moment matrix condition $\mathbf{M}_d(\mathbf{y}) \succeq 0$, these localizing matrix constraints constitute the semidefinite programming relaxations of order $d$ of Lasserre's hierarchy for polynomial optimization~\cite{lasserre_introduction_2015}. 

\clearpage
We now have all the necessary tools to derive a finite-dimensional relaxation of the measure program stated in~\eqref{eq:measure-program}.
To this end, let us introduce the truncated moment vectors associated with the measures $\mu, \mu_0, \mu_{\mathrm{T}}$, and $\mu_{\mathrm{a}}$, respectively:
\begin{align*}
	\begin{aligned}
		y_{(\bm{\alpha}, \beta)} &= \langle \mathbf{x}^{\bm{\alpha }}t^{\beta}, \mu \rangle, \\
		(y_0)_{(\bm{\alpha})} &= \langle \mathbf{x}^{\bm{\alpha}} , \mu_0 \rangle, \\
		(y_{\mathrm{T}})_{(\bm{\alpha}, \beta)} &= \langle \mathbf{x}^{\bm{\alpha}} t^{\beta}, \mu_{\mathrm{T}} \rangle, \\
		(y_{\mathrm{a}})_{(\bm{\alpha})} &= \langle \mathbf{x}^{\bm{\alpha}}, \mu_{\mathrm{a}} \rangle,
	\end{aligned}
\end{align*}
where the multi-indices $(\bm{\alpha}, \beta)$ are restricted to orders such that the corresponding moments are defined up to the relaxation order $d\in \mathbb{N}$, i.e.,
$
	|\bm{\alpha}| + |\beta| \leq d
$.
With these definitions in place, we can now formulate a hierarchy of semidefinite relaxations for the measure program~\eqref{eq:measure-program}. 

For each relaxation order $d\in\mathbb{N}$, we obtain the moment program
\begin{subequations}
	\label{eq:moment-program}
	\begin{align}
		P^\star_d=\sup_{\mathbf{y}, \mathbf{y}_0, \mathbf{y}_\mathrm{T}, \mathbf{y}_\mathrm{a}} \quad & L_{\mathbf{y}_\mathrm{a}}(1) \\
		\mathrm{s.t.} \qquad & \forall (\bm{\alpha},\beta) \in \mathbb{N}_{\leq 2d-d^\circ f }^{n+1}: \quad \operatorname{Liouv}_{\bm{\alpha}, \beta}(\mathbf{y}_0, \mathbf{y}, \mathbf{y}_\mathrm{T}) = 0  \label{seq:liouv-moment}\\
		& \forall \bm{\alpha} \in \bm{A}: \quad  {L}_{\mathbf{y}_0}(\mathbf{x}^{\bm{\alpha}}) = b_{\bm{\alpha}}, \label{seq:initial_moments}\\
		& {L}_{\mathbf{y}_0}(1) = 1, \\
		& \mathbf{M}_d(\mathbf{y}_\mathrm{T}-\mathbf{y}_\mathrm{a}) \succeq 0, \\
		& \mathbf{M}_d(\mathbf{y}_\mathrm{T}) \succeq 0, \, \mathbf{M}_d(\mathbf{y}_\mathrm{a}) \succeq 0, \, \mathbf{M}_d(\mathbf{y}_0) \succeq 0, \, \mathbf{M}_d(\mathbf{y}) \succeq 0, \\
		& \forall i\in[m_{\bm{T}}]: \quad \mathbf{M}_{d-d^\circ g_i^{\bm{T}}}(\mathbf{y}_\mathrm{T}| g_i^{\bm{T}}), \, \mathbf{M}_{d-d^\circ g_i^{\bm{T}}}(\mathbf{y}| g_i^{\bm{T}}) \succeq 0, \,  \\
		& \forall i\in[m_{\bm{X}}]: \quad \mathbf{M}_{d-d^\circ g_i^{\bm{X}}}(\mathbf{y}_\mathrm{T}| g_i^{\bm{X}}), \, \mathbf{M}_{d-d^\circ g_i^{\bm{X}}}(\mathbf{y}| g_i^{\bm{X}}) \succeq 0,\\
		& \forall i\in [m_{\bm{K}}]: \quad \mathbf{M}_{d-d^\circ g_i^{\bm{K}}}(\mathbf{y}_\mathrm{a}| g_i^{\bm{K}}) \succeq 0,  \\
		& \forall i\in [m_{\bm{X}_0}]: \quad \mathbf{M}_{d-d^\circ g_i^{\bm{X}_0}}(\mathbf{y}_0| g_i^{\bm{X}_0}) \succeq 0, \label{seq:constraint_mu0_moments} 
	\end{align}
\end{subequations}
where $d^\circ g = \lceil \operatorname{deg} g/2 \rceil$.

In~\eqref{seq:liouv-moment}, $\operatorname{Liouv}_{\bm{\alpha},\beta}(\mathbf{y}_0, \mathbf{y}, \mathbf{y}_\mathrm{T})=0$ encodes the moment constraints derived from the Liouville equation. Concretely, this condition reads
\begin{align}
	\langle \mathbf{x}^{\bm{\alpha}} t^{\beta}, \mu_\mathrm{T} \rangle -\langle \mathbf{x}^{\bm{\alpha}} t^{\beta}, \delta_0 \otimes \mu_0\rangle - \langle \mathcal{L}_f(\mathbf{x}^{\bm{\alpha}} t^{\beta}), \mu\rangle = 0,
\end{align} 
which, when expressed in terms of the corresponding moment sequences $\mathbf{y}_0,\mathbf{y},\mathbf{y}_\mathrm{T}$ yields a linear equality constraint in the moments.

\begin{remark}	
	The program~\eqref{eq:moment-program} due to the relaxation applied in its conceivement, provides upper bounds on~\eqref{eq:measure-program}. Also, it establishes a nonincreasing hierarchy of SDPs, because for higher relaxation orders, the feasible set shrinks. Under the same assumptions as in Theorem~\ref{theorem:no_gap}, namely, feasibility of~\eqref{eq:measure-program} and satisfaction of Assumption~\ref{ass:compactness}, this relaxed hierarchy converges monotonically to the optimal value of~\eqref{eq:measure-program}. See~\cite[Theorem 1]{lasserre_moments_2010} for more details.
\end{remark}

\subsection{Sum-of-squares program}%
\label{sec:sos_program}

To render the infinite-dimensional functional program~\eqref{eq:functional_program} computationally tractable, as mentioned at the beginning of Sect.~\ref{sec:finite_dimensional}, we restrict our attention to polynomials. Namely, we replace the nonnegativity constraints over an arbitrary set, to check nonnegativity of polynomials over basic semialgebraic sets. By replacing these conditions to sufficient conditions based on sum-of-squares decompositions, we obtain a strengthened version of~\eqref{eq:functional_program} with finite-dimensional SDP. 

\begin{remark}
	In practice, primal-dual interior-point methods for SDP solve the moment program~\eqref{eq:moment-program} and the SOS program~\eqref{eq:sos_program} simultaneously as they form a primal-dual pair. While either formulation suffices as input to an SDP solver, engineering applications tend to favor the functional (SOS) formulation due to its direct interpretability in terms of barrier functions and Lyapunov-like certificates. This preference motivates the presentation of the SOS program in this section. 
\end{remark}

\vspace{0.2cm}

We recall the standard SOS characterization: a polynomial $p\in\mathbb{R}[\mathbf{x}]$ belongs to $\Sigma[\mathbf{x}]$ if and only if there exist a positive semidefinite matrix $\mathbf{Q}\succeq 0$ and a vector of polynomials $\xi\in\mathbb{R}[\mathbf{x}]^m$ satisfying 
\begin{align*}
	p = \xi^\top \mathbf{Q} \xi.
\end{align*}

Thus, verifying that a polynomial is SOS is equivalent to solving a set of linear equality constraints coupled with a positive semidefinite constraint $\mathbf{Q}\succeq 0$; this is precisely a semidefinite program.

However, the constraints in~\eqref{eq:functional_program} require nonnegativity of polynomials over basic semialgebraic sets, rather than global nonnegativity. To handle, this we employ the well-known \textit{Positivstellensatz} certificate. Among the various available certificates, we adopt Putinar's Positivstellensatz due to its favourable scalability in practical implementations.  
Although Putinar's theorem establishes an equivalence under the Archimedean condition~\cite{lasserre_moments_2010}, for our purposes it suffices to use it as a sufficient condition: if the required SOS decomposition exists, then the desired nonnegativity over a semialgebraic set is guaranteed. This sufficient condition is the generalized $\mathcal{S}$-procedure~\cite{Tan2006}, which we state below.

\begin{theorem}(Generalized $\mathcal{S}$-procedure)%
	\label{theorem:generalized_S_procedure}
	Let $\bm{K}$ be a basic semialgebraic set defined by the polynomials $g_i \in \mathbb{R}[\mathbf{x}]$ as
	\begin{align*}
		\bm{K}=\{ \mathbf{x} \in \mathbb{R}^n  :  \forall  i \in [m], \, g_i(\mathbf{x}) \geq 0\}.
	\end{align*}
	If there exist $\sigma_i\in\Sigma[\mathbf{x}]$ for $\, i= 1,\ldots, m$, satisfying
	\begin{align*}
		h(\mathbf{x})-\sum_{i\in[m]} \sigma_i(\mathbf{x}) g_i(\mathbf{x}) \in \Sigma[\mathbf{x}],
	\end{align*}
	then the polynomial $h$ is nonnegative on $\bm{K}$.
\end{theorem}

To ease the use of Theorem~\ref{theorem:generalized_S_procedure}, we introduce the \textit{quadratic module} $QM(\bm{K}) \subset \mathbb{R}[\mathbf{x}]$ associated with $g_1, \ldots, g_m \in \mathbb{R}[\mathbf{x}]$ as
\begin{align*}
	QM(\bm{K}) = \left\{ \sigma_0(\mathbf{x}) +\sum_{j=1}^m \sigma_j(\mathbf{x}) g_j(\mathbf{x}) : \sigma_j \in \Sigma[\mathbf{x}] \right\}.
\end{align*}
To establish the hierarchy of finite dimensional programs, we introduce the notion of finite dimensional quadratic module. The degree $d$ quadratic module denotes
\begin{align*}
	QM_{d}(\bm{K}) = \left\{\sigma_0(\mathbf{x}) +\sum_{j=1}^m \sigma_j(\mathbf{x}) g_j(\mathbf{x}) : \sigma_0 \in \Sigma_d[\mathbf{x}], \, \sigma_j \in \Sigma_{d-\operatorname{deg}g_j}[\mathbf{x}]   \right\}.
\end{align*}


By a direct application of the generalized $\mathcal{S}$-procedure on the functional program~\eqref{eq:functional_program}, we get
\begin{subequations}%
	\label{eq:sos_program}
	\begin{align}
		D_d^\star= \min_{\omega, \nu, \bm{\gamma}} \quad & - \nu-\sum_{\bm{\alpha} \in \bm{A}} \gamma_{\bm{\alpha}} b_{\bm{\alpha}}  \\
		\text{s.t.} \quad & 
		 \omega-1 \in QM_d(\bm{T}\times \bm{K}), \label{eq:sos_cons_1}\\
		&  \omega \in QM_d(\bm{T}\times \bm{X}), \label{eq:sos_cons_2}\\
		&  - \mathcal{L}_f \omega \in QM_d(\bm{T} \times \bm{X}), \label{eq:sos_cons_3}\\
		&  - \omega(\mathbf{x}, 0) - \nu - \sum_{\bm{\alpha}\in \bm{A}} \gamma_{\bm{\alpha}} \mathbf{x}^{\bm{\alpha}} \in QM_d(\bm{X}_0). \label{eq:sos_cons_4} 
	\end{align}
\end{subequations}

The use of SOS certificates in program~\eqref{eq:sos_program} introduces a strengthening of the original functional program~\eqref{eq:functional_program}, in the sense that the SOS conditions are sufficient but not necessary for the local nonnegativity constraints. Consequently, for a fixed relaxation order $d$ the SOS program~\eqref{eq:sos_program} constitutes a restriction of~\eqref{eq:functional_program}, and therefore its optimal value provides a provable upper bound on the worst-case unsafe probability of Problem~\ref{prob:main}. 
\begin{remark}
	As the relaxation order increases, the gap between the SOS certificate and the true nonnegativity condition diminishes; under the Archimedean condition, satisfied by the compactness Assumptions~\ref{ass:all_sets}\textendash\ref{ass:unsafe_set}, convergence to the optimal cost of the functional program~\eqref{eq:functional_program} is guaranteed in the limit $d\to \infty$, i.e., $D_d^\star \downarrow D^\star$ as $d\to \infty$~\cite[Chapter 4]{lasserre_moments_2010}.
\end{remark}

\subsection{Ambiguity sets}%
\label{sec:ambiguity}

Instead of fixing the prescribed moments to exact values, we consider the more general setting where the moment vector belongs to a prescribed uncertainty set. Specifically, we assume that the noncentral moments of interest are known only to lie within a bounded region, which we denote as ambiguity set, for instance, a box constraint on the first- and second-order noncentral moments. This relaxation introduces an additional layer of uncertainty into Problem~\ref{prob:main}, which leads to a more conservative (i.e., larger) worst-case probability, yet provides a more robust quantification of safety under moment uncertainty. 

When the ambiguity set is defined by box constraints, these can be incorporated directly into the measure program~\eqref{eq:measure-program} and, consequently, into the moment program~\eqref{eq:moment-program} without destroying convexity. To this end, we treat the moment vector 
\begin{align*}
	\mathbf{b}=(b_{\bm{\alpha}})_{\bm{\alpha} \in \bm{A}}
\end{align*}
as a decision variable and impose the component-wise bounds
\begin{align*}
	\underline{\mathbf{b}} \leq \mathbf{b} \leq \overline{\mathbf{b}},
\end{align*}
where $\underline{\mathbf{b}}=(\underline{b}_{\bm{\alpha}})_{\bm{\alpha}\in\bm{A}}$ and $\overline{\mathbf{b}}=(\overline{b}_{\bm{\alpha}})_{\bm{\alpha}\in\bm{A}}$ denote the lower and upper bounds on the noncentral moments indexed by $\bm{\alpha} \in \bm{A}$, respectively.

However, incorporating these box constraints directly into the functional program~\eqref{eq:functional_program} and its associated SOS reformulation~\eqref{eq:sos_program} is not immediate, as the objective functional in~\eqref{eq:cost_functional} becomes bilinear in the moment variables and the dual multipliers. In the special case of box constraints, this difficulty can be circumvented via an algebraic reformulation that preserves convexity. Specifically, we replace the original cost functional with the following expression:
\begin{align*}
	- \nu-\sum_{\bm{\alpha} \in \bm{A}} \frac{\overline{b}_{\bm{\alpha}} + \underline{b}_{\bm{\alpha}}}{2} \gamma_{\bm{\alpha}} + \sum_{\bm{\alpha}\in \bm{A}} \frac{\overline{b}_{\bm{\alpha}}-\underline{b}_{\bm{\alpha}}}{2} |\gamma_{\bm{\alpha}}| ,
\end{align*}
where $|\gamma_{\bm{\alpha}}|$ denotes the absolute value of $\gamma_{\bm{\alpha}}$. 

To handle the absolute value terms, we introduce the auxiliary variable $\mathbf{t}=(t_{\bm{\alpha}})_{\bm{\alpha}\in\bm{A}}$ and the linear inequality constraints: $t_{\bm{\alpha}} \geq \gamma_{\bm{\alpha}}$ and $t_{\bm{\alpha}} \geq -\gamma_{\bm{\alpha}}$. Thus, in the SOS programming framework, the optimization with moment uncertainty takes the following convex form:
\begin{subequations}
	\begin{align}
		\min_{\omega, \nu, \bm{\gamma}, \mathbf{t}} \quad & -\nu - \frac{1}{2}\sum_{\bm{\alpha} \in \bm{A}} (\overline{b}_{\bm{\alpha}} + \underline{b}_{\bm{\alpha}}) \gamma_{\bm{\alpha}} + (\overline{b}_{\bm{\alpha}}-\underline{b}_{\bm{\alpha}}) t_{\bm{\alpha}}  \\
		\text{s.t.} \quad & \eqref{eq:sos_cons_1}\textendash\eqref{eq:sos_cons_4}, \\
		& \forall \bm{\alpha} \in \bm{A}, \quad t_{\bm{\alpha}} \geq \gamma_{\bm{\alpha}} \, \, \,  \wedge \, \, \, t_{\bm{\alpha}} \geq -\gamma_{\bm{\alpha}}. \label{eq:additional_cons}
	\end{align}
\end{subequations}

Note that considering this additional uncertainty adds a potential increase in computation time, since we required to add additional decision variables, namely $\mathbf{t}$, and linear inequality constraints~\eqref{eq:additional_cons}.
The impact of this ambiguity on the resulting worst-case bounds is illustrated in the next section. 

\clearpage
\section{Numerical examples}%
\label{sec:application}

In this section, we illustrate the proposed methods through numerical examples. All simulations are conducted in \textsc{MATLAB} implemented using the open-source toolbox \textsf{Ca$\mathsf{\Sigma}$oS}~\cite{cunis_casigmaos_2024} that handles the parsing to the low-level solver. The semidefinite programming solver used was \textsc{MOSEK}~\cite{noauthor_mosek_2019} and the results were obtained on a 12th Gen Intel Core i5-1235U CPU with 16.0 GB of RAM.

\subsection{Academic example}
	\label{ex:1}

	Consider the one-dimensional polynomial dynamical system 
	\begin{align}
		\label{eq:academic_example}
		\dot{x}(t)=0.5x(1-x^2),
	\end{align}
	evolving on the time-window $\bm{T}=[0,1]$. This is the same system used in Fig.~\ref{fig:visualize-liouv} in Sect.~\ref{sec:liouville_eq}.
	The initial set is $\bm{X}_0 = [-0.5, 0.5]$ and we consider $\bm{X}=[-2, 2]$. The region of interest is $\bm{K}=[0.5,1]$.
	
	
	\subsubsection*{Case 1. Known first- and second-order moments} 

	We first assume that the first two noncentral moments of the initial distribution are known exactly with
	\begin{align*}
		b_{1} = 0 \quad \text{and} \quad b_2 = 0.05.
	\end{align*}
	 Fig.~\ref{fig:ex1_stats} presents the upper bounds on the worst-case probability obtained by applying the SOS-based program~\eqref{eq:sos_program} to the dynamical system~\eqref{eq:academic_example}, with the specified state space $\bm{X}$, initial set $\bm{X}_0$, time horizon $\bm{T}$, and moment information $b_1, b_2$.

	\begin{figure}[H]
		\centering
		\includegraphics[width=0.69\textwidth]{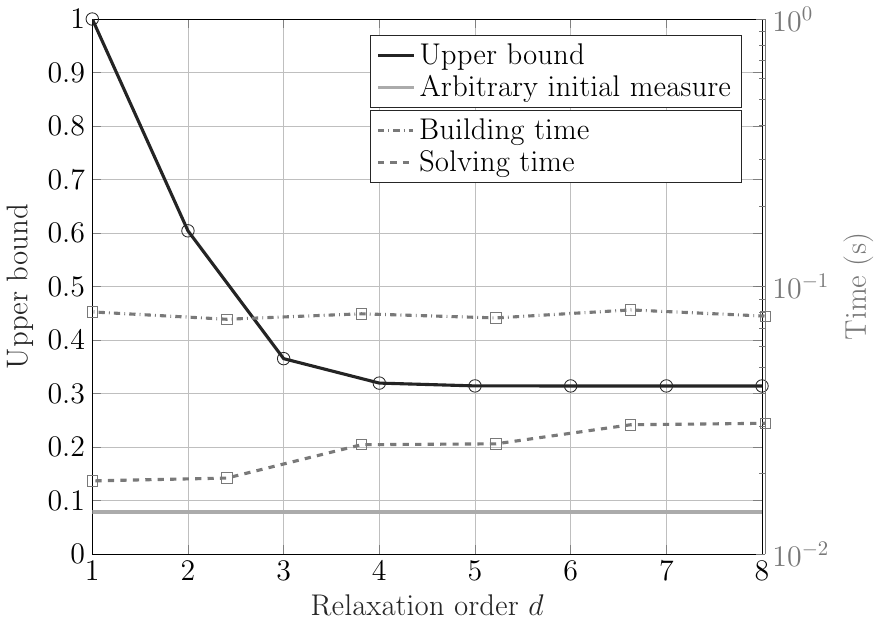}
		\caption{Upper bounds on the worst-case probability for the academic example in Sect.~\ref{ex:1} with known first- and second-order moments ($b_1$ and $b_2$, respectively).}%
		\label{fig:ex1_stats}
	\end{figure}

	As shown in Fig.~\ref{fig:ex1_stats}, the sequence of upper bounds $P_d^\star$ stabilizes around $0.31$, with a relaxation order $d=4$. Since this problem involves a one-dimensional system, the computation required to build and solve the underlying SDP is negligible and does not increase significantly over the tested relaxation orders.  

	As a sanity check, we have performed a Monte Carlo simulation with $10,000$ samples over the time interval $\bm{T}$, discretized with a step-size $dt=0.001$, an arbitrary initial distribution that follows the moment constraints. Specifically, we chose the polynomial density function 
	\begin{align}
		\label{eq:polynomia_density_function}
		p(x) = \begin{cases}
			-6x^2+1.5 & \text{for} \, x \in \bm{X}_0, \\
			0 & \text{otherwise}.
		\end{cases}
	\end{align}

	The estimated maximum probability attained for the system to be in $\bm{K}$ for some instance $t\in\bm{T}$ using the admissible density function~\eqref{eq:polynomia_density_function} was 0.0794. This value is significantly lower than the value attained by the finite-dimensional convex problem used. This stems from the fact that the problem solved looks for the worst-case $\mu_0$, while the one used in the Monte Carlo simulation is just an admissible $\mu_0$. In general, a randomly picked $\mu_0$ consistent with the moment constraints is unlikely to attain the worst-case probability, but it does provide a coarse lower bound on this value.
	More details on the Monte Carlo simulation, please see Appendix~\ref{sec:appendix_B}.

	\begin{remark}
		To the best of the authors' knowledge, no existing method searches for the worst-case initial measure $\mu_0$ that satisfies moment constraints while maximizing the probability that the system enters the set $\bm{K}$ at some time $t\in\bm{T}$. For this reason, we rely on an arbitrary initial measure that is guaranteed to satisfy the moment constraints, but is not guaranteed to be the worst-case one.
	\end{remark}

	\subsubsection*{Case 2. Ambiguity sets}

	To further investigate the effect of uncertainty in the moments, we now introduce ambiguity sets for the first and second moments 
	\begin{align*}
		\begin{aligned}
			b_1 &\in \left[\overline{b}_1-\kappa\Delta_{b_1}, \, \overline{b}_1+\kappa\Delta_{b_1}\right] \\
			b_2 &\in \left[\overline{b}_2-\kappa\Delta_{b_2}, \, \overline{b}_2+\kappa\Delta_{b_2}\right]
		\end{aligned}
	\end{align*}
	with 
	\begin{equation*}
		\overline{b}_1=0, \quad \overline{b}_2=0.05,
	\end{equation*}
	and
	\begin{equation*}
		\Delta_{b_1}=0.01, \quad \Delta_{b_2}=0.001,
	\end{equation*}
	for $\kappa\in \{0, 2, 5, 8, 12\}$. The results of applying the finite-dimensional program~\eqref{eq:sos_program} are shown in Fig.~\ref{fig:ex1_stats_amb}.

	\begin{figure}
		\centering
		\includegraphics[width=0.75\textwidth]{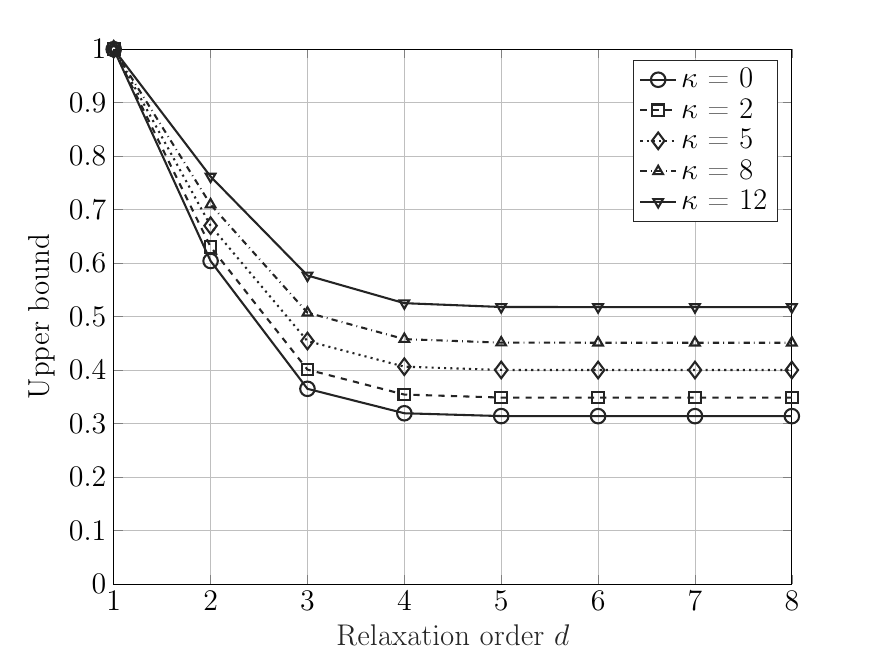}
		\caption{Upper bounds on the worst-case probability for different levels of ambiguity in the moment information, namely the first- and second-order moments, for the academic example in Sect.~\ref{ex:1}.}%
		\label{fig:ex1_stats_amb}
	\end{figure}

	Fig.~\ref{fig:ex1_stats_amb} shows the expected behavior of considering the ambiguity sets: for a larger ambiguity set, the worst case probability is equal or larger than considering a smaller ambiguity set; additionally, this result shows that we can successfully analyze the effect of small uncertainties in the supposed known noncentral moments.  
	
	Moreover, in this simple case, the changes to the finite-dimensional problem to handle the ambiguity sets did not incur a significant addition to the computational time, although it is expected this to become an issue for larger and more complex systems. Compared to the nominal case, we require adding new decision variables and linear inequality constraints.

\subsection{Example: Object in orbit}%
	\label{ex:3}
	
	To illustrate the proposed framework in a setting of practical relevance, we examine the motion of an uncontrolled object in Earth orbit, as it traverses the vicinity of a functioning satellite. Concretely, we seek to evaluate the probability that, under the current orbital state, the object enters a prescribed ball centered at the satellite; a region we henceforth term the unsafe set, denoted by $\bm{K}$ in accordance with the preceding development.
	
	In the Earth-Centered Inertial reference frame,  the position vector $\mathbf{r}=(r_x,r_y,r_z)\in\mathbb{R}^3$ and velocity vector $\mathbf{v} = (v_x, v_y, v_z)\in\mathbb{R}^3$ of an in-orbit object satisfy the classical two-body dynamics~\cite{casella_high-accuracy_2007}
\begin{align}
	\label{eq:2nd-dynamics}
	\frac{\mathrm{d}^2\mathbf{r}}{\mathrm{d}t^2} = - \varpi  \mathbf{r} \| \mathbf{r}\|^{-3}_2 
\end{align}
where $\varpi=3.986\cdot 10^{14}~\mathrm{m^3\, s^{-2}}$ is the standard gravitational parameter of the Earth, and $\|\cdot\|_2$ denotes the Euclidean norm. 
With the state representation $\mathbf{x}=(\mathbf{r}, \mathbf{v})$, we write~\eqref{eq:2nd-dynamics} in the form 
\begin{align}
	\label{eq:ode}
	\dot{\mathbf{x}}(t) = f(\mathbf{x}) = \begin{bmatrix}
		\mathbf{v} \\
		- \varpi  \mathbf{r} \| \mathbf{r}\|^{-3}_2
	\end{bmatrix}.
\end{align}

Since the programs developed in Sect.~\ref{sec:finite_dimensional} presupposes polynomial dynamics, we approximate~\eqref{eq:ode} via Taylor expansion about the mean initial state. Over extended time horizons, the Taylor expansion naturally incurs significant error. A more accurate approximation could be obtained via a least-squares fit over samples drawn from the reachable region over the time horizon of interest. For the sake of simplicity and ease of exposition, we restrict ourselves to the coarser Taylor approximation. 
For a specified state $\mathbf{x}_0\in\mathbb{R}^n$, we obtain a polynomial approximation of~\eqref{eq:ode} of degree $d$ by
\begin{align}
	\label{eq:poly-approx-taylor}
	\dot{x}_i (t) \approx \sum_{|\bm{\alpha}| \leq d} \frac{1}{\bm{\alpha}!}D^{\bm{\alpha}} f_i(\mathbf{x}_0) (\mathbf{x}-\mathbf{x}_0)^{\bm{\alpha}} \quad \text{for} \quad i\in[n]
\end{align}
where
\begin{equation*}
	D^{\bm{\alpha}} f_i := \frac{\partial^{\alpha_1 + \cdots + \alpha_n} f_i}{\partial x_1^{\alpha_1} \cdots \partial x_n^{\alpha_n}}.
\end{equation*}

As a proof of concept of the programs presented in Sect.~\ref{sec:finite_dimensional}, we restrict ourselves to the planar (two-dimensional) case, so that $\mathbf{x}=(\mathbf{r}, \mathbf{v})\in\mathbb{R}^4$. The Taylor expansion is performed about the initial condition $\mathbf{x}_0 = (\mathbf{r}_0, \mathbf{v}_0)$ with 
		$\mathbf{r}_0 =(6828\cdot 10^3, \, 0 )~\mathrm{m}$ 
		and  
		$\mathbf{v}_0 = (0 , \, 5.4\cdot 10^3 )~\mathrm{m\, s^{-1}}$. 
The resulting second-order Taylor approximation reads

\begin{align*}
	\dot{\mathbf{x}}(t) \approx \begin{bmatrix}
		v_x \\
		v_y \\
		-51.2981 + 1.00172\cdot 10^{-5}r_x - 5.50154\cdot 10^{-13}r_x^2 + 2.75077\cdot 10^{-13}r_y^2\\
		-5.0086\cdot 10^{-6}r_y + 5.50154\cdot 10^{-13}r_x r_y
	\end{bmatrix}.
\end{align*}
\vspace{0.2cm}

To mitigate ill-conditioning in the underlying SDP, we applied a scaling transformation to the states so that all sets of interest lie within the unit hypercube $[-1,1]^n$. The scaling vector is defined as
	$S_{\mathbf{x}} = 
	(
		7\cdot10^6, \, 
		3\cdot 10^6, \,
		5\cdot 10^3, \,
		5\cdot 10^3
	)$
such that the original state vector $\mathbf{x}$ relates to the scaled state $\tilde{\mathbf{x}}$ via the Hadamard (element-wise) product by $\mathbf{x} = \tilde{\mathbf{x}} \odot S_{\mathbf{x}}$.
Furthermore, the time horizon is normalized to $[0,1]$ from $\bm{T}=[0, \, 500]~\mathrm{s}$. Under these transformations, the scaled second-order polynomial dynamics take the form
\begin{align}
	\label{eq:second_order_approximation}
	\dot{\tilde{\mathbf{x}}}(t) \approx\begin{bmatrix}
		0.357143 \tilde{v}_x \\
		0.833333 \tilde{v}_y \\
		-5.12981 + 7.01205\tilde{r}_x - 2.69576\tilde{r}_1^2 + 0.247569\tilde{r}_y^2\\
		-1.50258\tilde{r}_y + 1.15532\tilde{r}_x \tilde{r}_y
	\end{bmatrix},
\end{align}
with $\tilde{\mathbf{x}}=(\tilde{r}_x, \tilde{r}_y, \tilde{v}_x, \tilde{v}_y)$.
\vspace{0.2cm}


With the baseline system and scaling established, we now proceed to investigate three distinct scenarios:
\begin{enumerate}
	\item \textbf{Fixed unsafe region with known moments:} We estimate the worst-case probability for a time-independent unsafe set $\bm{K}$, assuming knowledge of the first-order moments and partial information about the second-order moments of the initial state distribution.

	\item \textbf{Moment ambiguity:} Extending the first scenario, we incorporate ambiguity in the specified moment information itself, thereby accounting for possible inaccuracies or uncertainties in the moment data.

	\item \textbf{Time-dependent unsafe region:} Finally, we consider a time-varying unsafe set $\bm{K}$, which can be interpreted, for instance, as the spatial region over which a satellite is expected to pass during the mission time window.
\end{enumerate}

In all three cases, we assume that the initial distribution has a bounded support $\bm{X}_0$, which, in the scaled variables, takes the form
\begin{align*}
	\tilde{\bm{X}}_0 = \{ \tilde{\mathbf{x}}\in\mathbb{R}^4 \, : \,  0.1-\|\tilde{\mathbf{x}}-\tilde{\mathbf{x}}_0\|^2_2 \geq 0\}.
\end{align*}
Additionally, to ensure compactness of all measure-representable moment sequences, we impose the further restriction that all measures have support contained in 
\begin{align*}
	\tilde{\bm{X}}=\{ \tilde{\mathbf{x}}\in\mathbb{R}^4 \, : \, 4-\|\tilde{\mathbf{x}}\|^2_2 \geq 0\}.
\end{align*} 

\subsubsection*{Case 1. Fixed unsafe region with known moments}

As a first safety analysis, similar to Case 1 in Sect.~\ref{ex:1}, we search for a (tight) upper bound on the worst-case probability for a time-independent unsafe set $\bm{K}$, assuming partially known first- and second-order moments of the initial state distribution $\mu_0$.

We assume that the mean and some second-order noncentral moments of the initial measure $\mu_0$ are perfectly known, taking the values 
\begin{align}
	\label{eq:hadamard}
	\begin{aligned}
		b_{1000} = 0.9754, \quad\quad &b_{0100} = 0.000,\\
		b_{0010} = 0.0000, \quad\quad &b_{0001} = 1.080,\\
		b_{2000} = 1.0015, \quad\quad &b_{0200} = 0.010.  
	\end{aligned}
\end{align}
The unsafe set of interest $\bm{K}$ considered here is defined in the scaled variables as
\begin{align*}
	\bm{K} = \{ \tilde{\mathbf{x}} \in \mathbb{R}^4 \, : \, 0.01 - \|\tilde{\mathbf{r}}-\tilde{\mathbf{r}}_{\bm{K}}\|^2_2 \geq 0\}
\end{align*}
with $\tilde{\mathbf{r}}_{\bm{K}} = (0.81, 0.84)$. Fig.~\ref{fig:ex1_case1_illustration} illustrates the scenario under study, where the region $\bm{K}$, the nominal trajectory and randomly sampled trajectories taken from an arbitrary distribution with support in $\bm{X}_0$ that satisfies the moments~\eqref{eq:hadamard} are depicted.
The results obtained using the scaled second-order approximation~\eqref{eq:second_order_approximation} in to the moment program~\eqref{eq:sos_program} for relaxation orders $3, 4, 5, 6, 7, 8$ are presented in Fig.~\ref{fig:ex3_stats}.

\begin{figure}[H]
	\centering
	\includegraphics[width=0.65\textwidth]{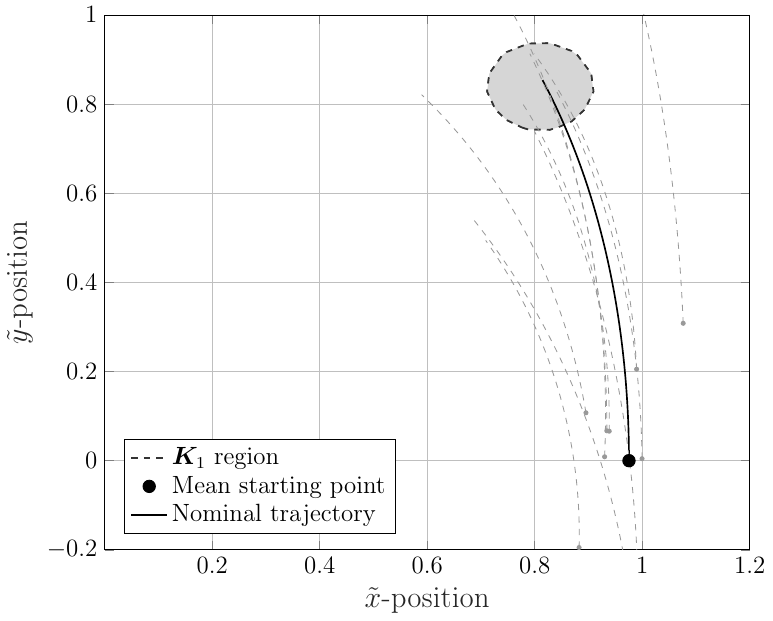}
	\caption{In-orbit safety verification illustration for Case 1 (Sect.~\ref{ex:3}). The nominal trajectory and sampled trajectories are plotted over the normalized time window $[0,1]$ along the scaled $(\tilde{x}, \tilde{y})$-position axes. The mean starting point and the $\bm{K}_1$ region are also indicated.}%
	\label{fig:ex1_case1_illustration}
\end{figure}

\begin{figure}[H]
	\centering
	\includegraphics[width=0.68\textwidth]{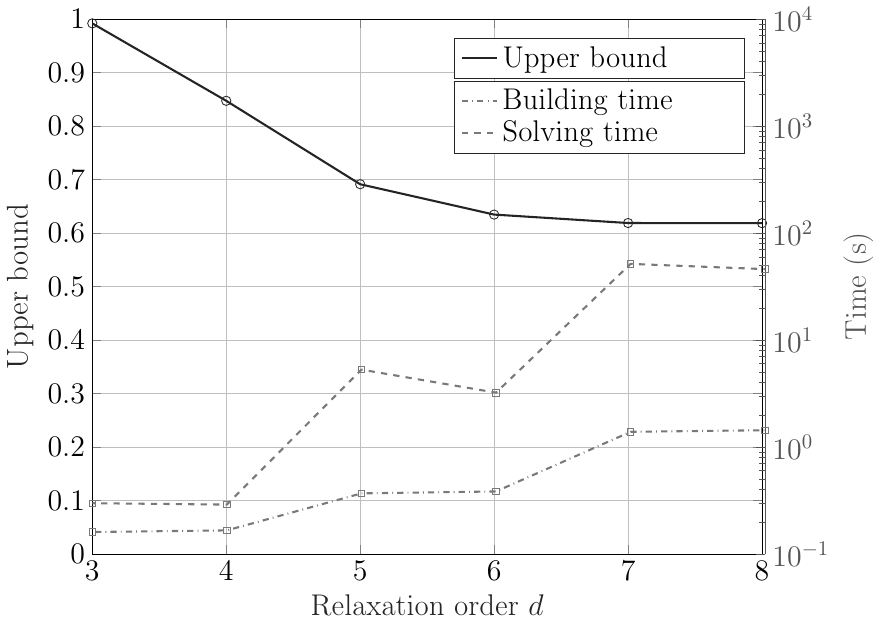}
	\caption{Upper bound on the worst-case probability for a time-independent unsafe region $\bm{K}$ under partially known noncentral moments (Case 1, Sect.~\ref{ex:3}). The left vertical axis shows the bound as a function of the relaxation order $d$, while the right vertical axis reports the associated building and solving times.}%
	\label{fig:ex3_stats}
\end{figure}

The following observations can be made from Fig.~\ref{fig:ex3_stats}:
increasing the relaxation order yields tighter estimates of the worst-case probability, seeming to start to stall at the relaxation order $d=7$, at around $0.6183$. However, this comes at a significant computational cost; although the building time of the underlying SDP is a bit more than 1 second, the solving time for this relaxation order is around $51.43$ seconds. If one would try even higher relaxation orders the computation would increase significantly. We can observe that the solving time evolves almost linearly in a logarithm scale.

Furthermore, we examined the program's behavior for three distinct regions of interest, each representing a small circular region in the position space. These regions are defined as:
\begin{align*}
	\begin{aligned}
		\bm{K}_1 &= \{ \tilde{\mathbf{x}} \in \mathbb{R}^4 \, : \, 0.01 - \|\tilde{\mathbf{r}}-\tilde{\mathbf{r}}_1\|^2_2 \geq 0 \} \\
		\bm{K}_2 &= \{ \tilde{\mathbf{x}} \in \mathbb{R}^4 \, : \, 0.01 - \|\tilde{\mathbf{r}}-\tilde{\mathbf{r}}_2\|^2_2 \geq 0 \} \\ 
		\bm{K}_3 &= \{ \tilde{\mathbf{x}} \in \mathbb{R}^4 \, : \, 0.01 - \|\tilde{\mathbf{r}}-\tilde{\mathbf{r}}_3\|^2_2 \geq 0 \}
	\end{aligned}
\end{align*}
with $\tilde{\mathbf{r}}_1=(0.81, \, 0.84)$, $\tilde{\mathbf{r}}_2=(1, \, 0.6)$, and $\tilde{\mathbf{r}}_3=(0.4, \, 0.6)$.
These regions are chosen to lie at different distances from the nominal trajectory, thereby allowing us to assess how the estimated upper bound on the worst-case probability vary with region of interest.
The results are summarized in Fig.~\ref{fig:multiple_Ks}.

\begin{figure}[H]
	\centering
	\includegraphics[width=0.75\textwidth, clip, trim={1.1cm, 0, 0, 0}]{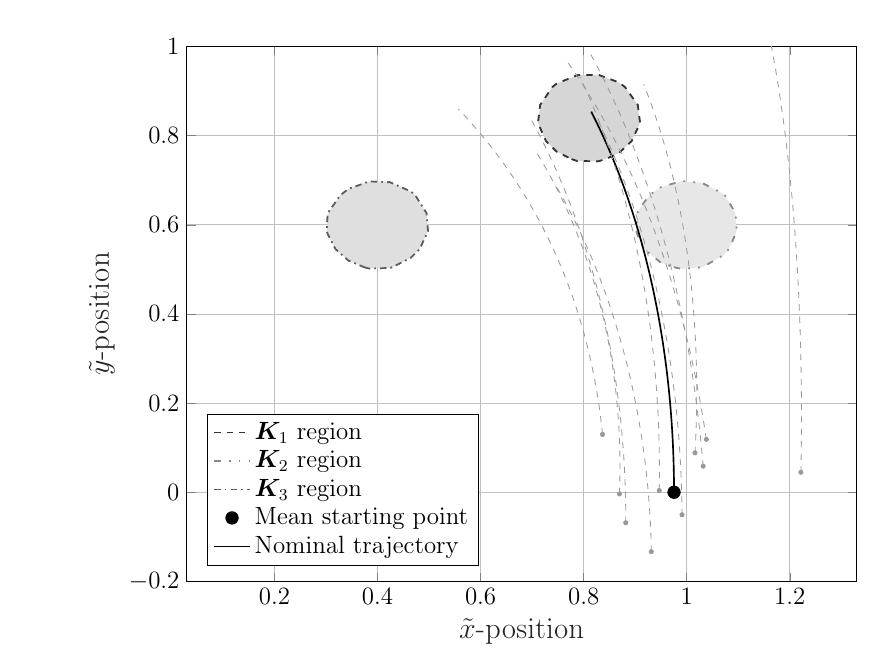}
	\caption{Illustration of the in-orbit safety verification for Case 1 (Sect.~\ref{ex:3}) with multiple unsafe regions. The nominal trajectory and sampled trajectories are plotted over the normalized time horizon $[0,1]$ in the scaled $(\tilde{x}, \tilde{y})$-position axes. The mean starting point and the regions $\bm{K}_1$, $\bm{K}_2$, and $\bm{K}_3$ are also marked.}%
	\label{fig:multiple_Ks_illustration}
\end{figure}
\begin{figure}[H]
	\centering
	\includegraphics[width=0.75\textwidth]{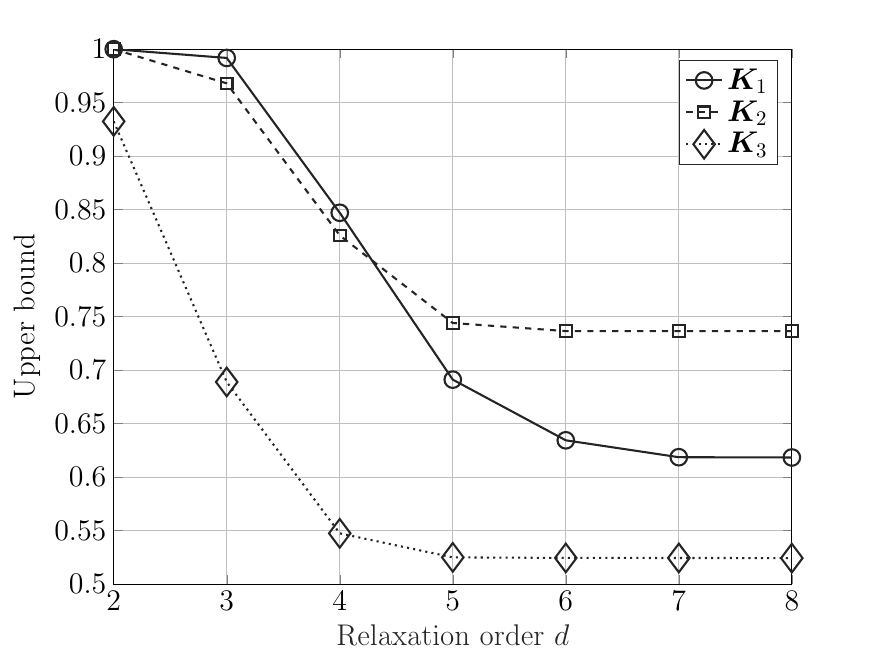}
	\caption{Upper bound on the worst-case probability for multiple time-independent unsafe regions $\bm{K}_1, \bm{K}_2,$ and $\bm{K}_3$ under partially known first and second-order noncentral moments (Case 1, Sect.~\ref{ex:3}). The bounds are plotted as a function of the relaxation order $d$.}%
	\label{fig:multiple_Ks}
\end{figure}

Several observations can be drawn from the results in Fig.~\ref{fig:multiple_Ks}. First, for the relaxation order $d=8$, the highest upper bound on the worst-case probability is achieved for $\bm{K}_2$ with $D_d^\star = 0.7364$; unsurprisingly, as this region lies close to the nominal trajectory. This aligns with the expectations: regions farther from the nominal trajectory generally yield lower probabilities.

A more striking result is observed for $\bm{K}_3$, which is located significantly away from the nominal trajectory in a sparsely visited area. Here, the estimated maximum probability drops substantially to $D_d^\star=0.5198$. However, this value is likely far from the true probability given the scenario; the bound on the probability may not be very tight, and the true value is expected to be considerably smaller.
This suggests the need to move to a higher level of the hierarchy. Yet, as the results from Fig.~\ref{fig:ex3_stats} indicate, doing so may impose a substantial additional computational burden, potentially rendering the approach unsuitable for real-time applications. Therefore, improved methods that minimize this overhead are desirable. 

\subsubsection*{Case 2. Moment ambiguity}

From a practical perspective, assuming perfectly known first- and second-order moments is too holistic, and might lead to undermining the safety of the system. To tackle this problem, similar to Case 2 in Sect.~\ref{ex:1}, we study the case where the moments are not perfectly known, but instead belong to some set, hence accounting for possible inaccuracies in the moments knowledge.

We consider the following arbitrarily chosen ambiguity sets
\begin{align*}
	\begin{aligned}
		b_{1000} &\in \left[\overline{b}_{1000}(1-\kappa), \, \overline{b}_{1000}(1+\kappa)\right] \\
		b_{0100} &\in \left[\overline{b}_{0100}(1-\kappa), \, \overline{b}_{0100}(1+\kappa)\right] \\
		b_{0010} &\in \left[\overline{b}_{0010}(1-\kappa), \, \overline{b}_{0010}(1+\kappa)\right] \\
		b_{0001} &\in \left[\overline{b}_{0001}(1-\kappa), \, \overline{b}_{0001}(1+\kappa)\right] \\
		b_{2000} &\in \left[\overline{b}_{2000}(1-\kappa), \, \overline{b}_{2000}(1+\kappa)\right] \\
		b_{0200} &\in \left[\overline{b}_{0200}(1-\kappa), \, \overline{b}_{0200}(1+\kappa)\right] \\
	\end{aligned}
\end{align*}
with 
\begin{align*}
	\begin{aligned}
		&\overline{b}_{1000} = 0.9754,
		\qquad \qquad &\overline{b}_{0100} = 0.00,
		\qquad \qquad &\overline{b}_{2000} = 1.0015,\\
		&\overline{b}_{0010} = 0.00,
		\qquad \qquad&\overline{b}_{0001} = 1.08,
		\qquad \qquad&\overline{b}_{0200} = 0.01,\\
	\end{aligned}
\end{align*}
for $\kappa\in \{0, 0.001, 0.002, 0.005, 0.008\}$. The results of applying the finite-dimensional program~\eqref{eq:sos_program} to system~\eqref{eq:second_order_approximation} considering $\bm{K}_1$ and all the 5 possible values of $\kappa$, i.e., the nominal case with zero ambiguity and 4 cases with increased ambiguity, is presented in Fig.~\ref{fig:ex3_ambiguity}.

\begin{figure}[H]
	\centering
	\includegraphics[width=0.75\textwidth]{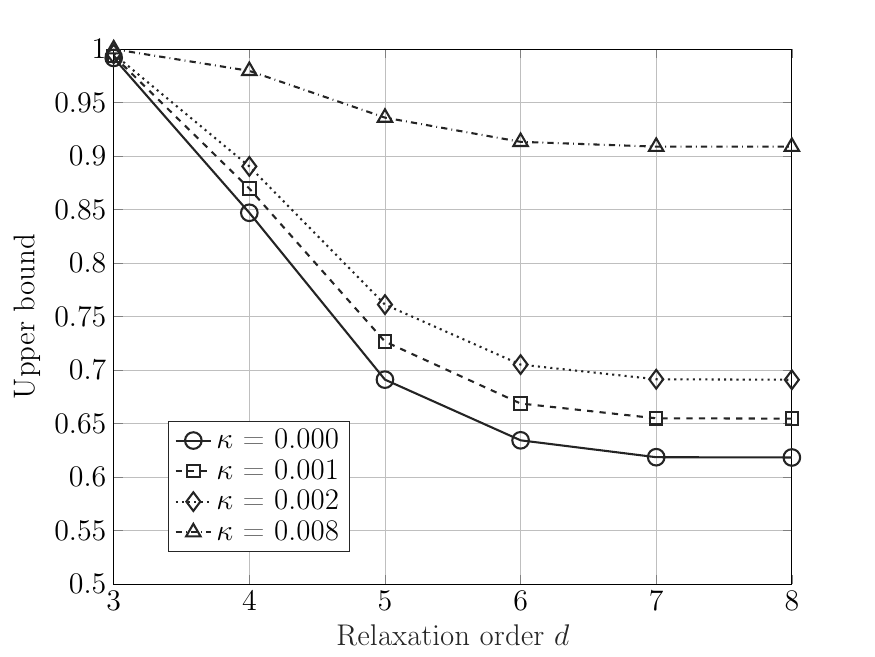}
	\caption{Upper bound on the worst-case probability for varying ambiguity set sizes, parameterized by $\kappa$ (Case 2, Sect.~\ref{ex:3}). The bounds are plotted over the relaxation order $d$.}%
	\label{fig:ex3_ambiguity}
\end{figure}

\clearpage
From Fig.~\ref{fig:ex3_ambiguity}, we observe that after the relaxation order $d=7$, the upper bound estimation through the finite-dimensional program seems to stabilize. Moreover, we can clearly see the influence of small uncertainties in the moments' information, and for the case of $\kappa=0.008$, the upper bound is reasonably high, at around $0.9089$. 

\subsubsection*{Case 3. Time-dependent unsafe region}

In this final case, we consider a more realistic settings in which the unsafe region $\bm{K}$ is not time-independent but rather it evolves in time. Specifically, we model $\bm{K}$ as a ball that moves according to a prescribed velocity, and in the scaled variables we define
\begin{align*}
	\bm{K} := \{ (\tilde{\mathbf{x}}, \tilde{t}) \in \mathbb{R}^4 \times \mathbb{R}_{\geq 0} : 0.005- \| \tilde{\mathbf{r}}-\tilde{\mathbf{r}}_0-\mathfrak{v}\tilde{t}\|_2^2 \geq 0\}
\end{align*}
with $\mathfrak{v}= (-0.1, \, -0.1)$ and $\tilde{\mathbf{r}}_0=(0.81,0.84)$. This setting is depicted in Fig.~\ref{fig:ex3_time_variant}, where three snapshots at normalized times $\tilde{t}\in\{0,0.5,1\}$ are shown.

\begin{figure}
	\centering
	\includegraphics[width=1\textwidth]{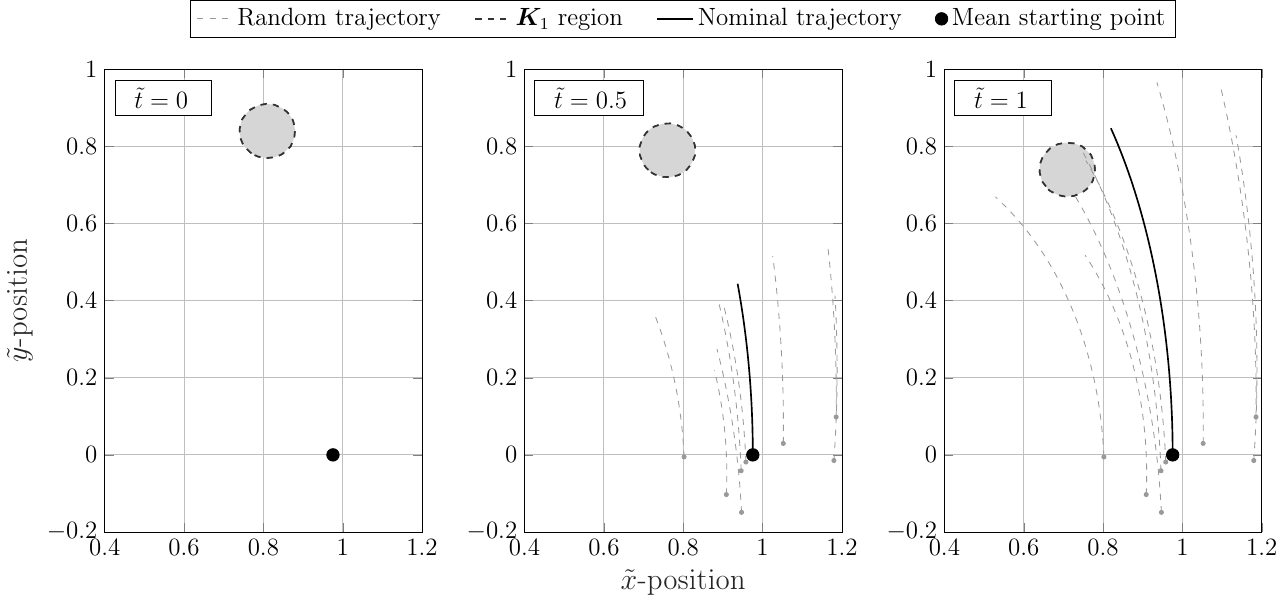}
	\caption{Illustration of the moving unsafe region $\bm{K}$ for the in-orbit safety verification problem (Case 3, Sect.~\ref{ex:3}). The region is shown as a ball translated over time according to $\tilde{\mathbf{r}}_0+\mathfrak{v}\tilde{t}$, with snapshots at $\tilde{t}=0,0.5,1$. Also shown are the nominal trajectory, random sampled trajectories, and the mean starting point of the initial measure.}%
	\label{fig:ex3_time_variant}
\end{figure}

To adapt the programs presented in the earlier sections, we extend the measure $\mu_\mathrm{a}$ so that its support lies in $\bm{K}\subset \bm{X} \times \bm{T}$, and we impose the modified constraint
\begin{align*}
	\mu_\mathrm{T} -\mu_\mathrm{a} \in \mathcal{M}_+(\bm{X}\times \bm{T}).	
\end{align*}
The results from applying the finite-dimensional program~\eqref{eq:sos_program} with these minor modifications are reported in Fig.~\ref{fig:moving-K}.
\begin{figure}[H]
	\centering
	\includegraphics[width=0.75\textwidth]{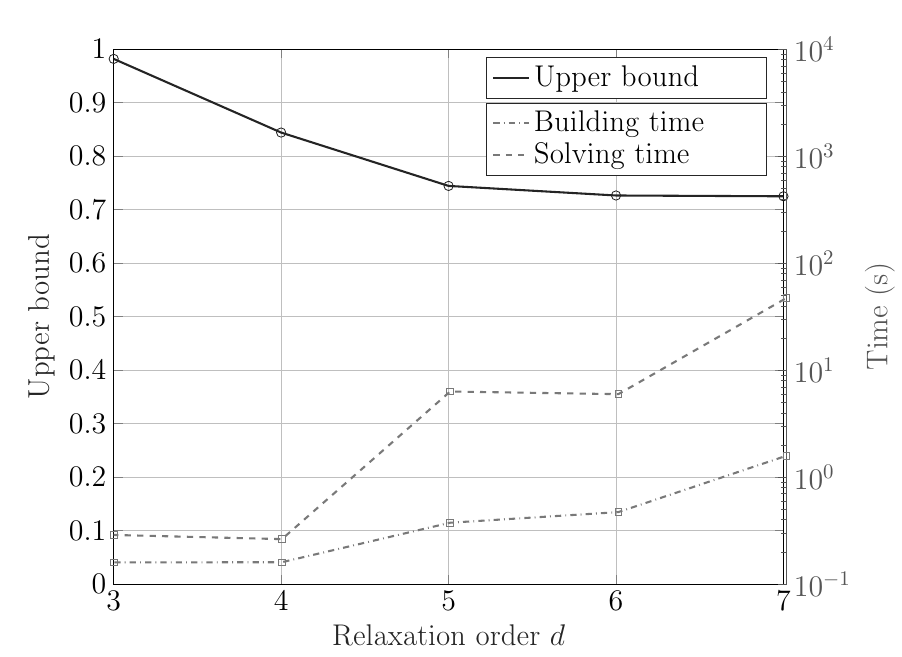}
	\caption{Upper bound on the worst-case probability for a time-varying unsafe region $\bm{K}$ under partially known noncentral moments (Case 3, Sect.~\ref{ex:3}). The left vertical axis shows the bound as a function of the relaxation order $d$, while the right vertical axis reports the associated building and solving times.}%
	\label{fig:moving-K}
\end{figure}

As in the time-independent case, the upper bound estimate reaches a plateau after relaxation order $d=7$, stabilizing at approximately $0.72$, and requiring only $47.3$ seconds of computation.

Beyond the numerical bound itself, the solution of the underlying SDP offers additional insight. In particular, we can extract the first-order pseudomoment with respect to time of the terminal measure obtained from~\eqref{eq:moment-program}, here denoted as $(\mu^\star_\mathrm{T})_d$ for each order $d$. 
This quantity is of special interest because, under the assumption that the worst-case probability is attained at a single instant, i.e., the extremal terminal measure takes the form $\delta_{t^\star} \otimes \nu$, with $\nu \in \mathcal{M}_+(\bm{X})$, this pseudomoment provides a direct estimate of the critical time $t^\star$.
As the hierarchy of relaxations converges, we have the limit
\begin{align*}
	\lim_{d\to \infty} \, \langle t, (\mu^\star_\mathrm{T})_d \rangle \underset{}{\longrightarrow} \langle t, \delta_{t^\star} \otimes \nu \rangle = t^\star,
\end{align*}
which justifies using this moment as an approximation of the worst-case time instance. For relaxation order $d=7$, the estimated value of $t^\star$ is $0.7180$. Thus, even when the bound stagnates, the moment data can reveal meaningful information about the scenario, rather than just the upper bound on the worst-case probability. 

\clearpage
\section{Conclusion}%
\label{sec:conclusion}

In this paper, we addressed the problem of estimating upper bounds on the worst-case probability that a dynamical system enters an unsafe region over a finite-time horizon, given partial information on the initial state in the form of noncentral moments. \replaced{This problem is particularly important in safety-critical systems such as in-orbit operations, where decision must be made ensuring the system is safe against any uncertain object.}{}
Leveraging measure-theoretic techniques, we formulate a convex infinite-dimensional linear programming problem over measures, along with its dual\textemdash an infinite-dimensional functional optimization program. We demonstrate that both problems can be relaxed into finite-dimensional semidefinite programs via Lasserre's moment-sum-of-squares hierarchy. Under mild regularity assumptions, we established convergence of the optimal values of these relaxed problems to the true worst-case probability of the original problem.

Furthermore, we extend our framework to accommodate uncertainty in the moment data itself and illustrate the practical implications of this extension through two simple numerical examples. Across all case studies, we observe that the upper-bound estimates stabilize on a reasonably small relaxation order, yielding both safe and informative assessments of system safety and risk.

Several directions warrant further investigation. First, for real-time risk assessment, warm-starting the SDP solvers\replaced{, if possible,}{} with prior solutions could reduce computation in sequential monitoring. Second, the programs developed require polynomial dynamics, but as seen in Sect.~\ref{ex:3} related to an object in-orbit, its own dynamics are not polynomial but rational, and required some polynomial approximation. These may compromise the rigor of the obtained bounds; in Appendix~\ref{sec:appendix_C} a procedure to reformulate the problem into the presented framework is shown, but it increases the size of the problem significantly, and some experiments have shown that we require a significantly large relaxation order to have some meaningful result. Future work should attempt to make this approach more tractable. Third, the optimization problems presented can be adapted for control synthesis, enabling policies that actively minimize the probability of unsafe encounters. Finally, complexity remains a major hurdle: high-order relaxations are computationally expensive.  
To address this, exploiting sparsity or adopting the sublevel Moment-SOS hierarchy from~\cite{chen_sublevel_2022-1} could yield tighter bounds with moderate computational increases, paving the way toward practical deployment.

\backmatter

\bmhead{Acknowledgements}

We would like to thank Jan Olucak for reviewing the paper and for the fruitful discussions that helped improve this paper.

\bmhead{Author contributions}

Renato Loureiro: Conceptualization, Methodology, Software, Investigation, Writing -- Orignal Draft. Torbj\o rn Cunis: Conceptualization, Writing -- review \& editing, and Supervision.


\bmhead{Conflict of interest} The authors declare no conflict of interest.

\bmhead{Code availability} The code is available in~\href{https://doi.org/10.18419/DARUS-6410}{https://doi.org/10.18419/DARUS-6410}.

\clearpage
\begin{appendices}

\section{Proof extension}%
\label{sec:appendix_A}

\begin{proof}[Complementary proof of~Theorem~\ref{theorem:functional}.]
	To obtain the dual program of~\eqref{eq:measure-program}, we associate Lagrange multipliers with each constraint in~\eqref{eq:measure-program}: $\omega \in \mathcal{C}(\bm{T}\times \bm{X})$ for the domination constraint~\eqref{eq:domination_constraint}; $\bm{\gamma}=(\gamma_{\alpha})_{\bm{\alpha}\in\bm{A}}\in \mathbb{R}^{|\bm{A}|}$ for each moment constraint $\langle \mu_0, \mathbf{x}^{\bm{\alpha}} \rangle = b_{\bm{\alpha}}, \, \bm{\alpha} \in \bm{A}$; $\nu\in \mathbb{R}$ for the mass normalization constraint $\langle \mu_0, 1 \rangle=1$. 
	
	The Liouville equation~\eqref{eq:liouville_short_free} is imposed weakly: for all test functions $\omega\in \mathcal{C}(\bm{T} \times \bm{X})$, 
	\begin{align*}
		\langle \mathcal{L}_f\omega, \mu \rangle = \langle \omega, \mu_\mathrm{T} - \delta_0 \otimes \mu_0 \rangle. 
	\end{align*}
	
	The Lagrangian $\mathfrak{L}$, after rearranging terms to isolate the measures $(\mu_\mathrm{a}, \mu_\mathrm{T}, \mu, \mu_0)$ in the duality pairing, reads
	\begin{align*}
		\begin{aligned}
			\mathfrak{L} = \langle 1 -\omega, \mu_\mathrm{a}\rangle + \langle \omega, \pi_\#^\mathbf{x} \mu_\mathrm{T} \rangle + \langle \mathcal{L}_f \omega, \mu \rangle + \left\langle \omega(\cdot,0)+\nu+\sum_{\bm{\alpha}\in\bm{A}} \gamma_{\bm{\alpha}} \mathbf{x}^{\bm{\alpha}}, \mu_0 \right\rangle
			-\nu - \sum_{\bm{\alpha}\in \bm{A}} \gamma_{\bm{\alpha}} b_{\bm{\alpha}}.
		\end{aligned}
	\end{align*}

	For the supremum over the nonnegative measures $(\mu_\mathrm{a}, \mu_\mathrm{T}, \mu, \mu_0)$ to be finite, the functions in each duality pairing must be nonpositive in the support of the corresponding measure in the duality pairing. This yields the dual feasibility conditions: $1-\omega\leq 0$ on $\bm{K} \times \bm{T}$; $\omega \geq 0$ and $\mathcal{L}_f\omega \leq 0$ on $\bm{T} \times \bm{X}$, and $\omega(\mathbf{x},0)+\nu+\sum_{\bm{\alpha} \in \bm{A}} \gamma_{\bm{\alpha}} b_{\bm{\alpha}} \leq 0$ on $\bm{X}_0$. 
	
	Under these constraints, the supremum of the Lagrangian over all admissible measures is 
	\begin{align*}
		\sup_{\mu_0, \mu_\mathrm{T}, \mu, \mu_\mathrm{a}} \mathfrak{L} = -\nu - \sum_{\bm{\alpha \in \bm{A}}} \gamma_{\bm{\alpha}} b_{\bm{\alpha}}.
	\end{align*}
	The dual problem therefore consists of minimizing this expression subject to the derived constraints, which is exactly the functional program~\eqref{eq:functional_program}.
\end{proof}

\section{Monte-Carlo}%
\label{sec:appendix_B}

Here we briefly describe the basic Monte Carlo simulation approach alongside its limitations. For more details, see for instance~\cite{tempo_randomized_2005}.

Let $(\bm{\Omega}, \bm{F}, \mu)$ be a probability space, and let $\mathbf{x}^1, \ldots, \mathbf{x}^s \in \mathbb{R}^n$ be $s\in \mathbb{N}$ independent samples drawn from $\mu$. Define the indicator functional 
\begin{align*}
	\mathbf{1}_{\bm{K}}:\mathbb{R}^n \to \{0,1\}, \quad \mathbf{x}\mapsto \begin{cases} 
		1 \quad \text{if} \quad \mathbf{x} \in \bm{K} \\
		0 \quad \text{otherwise}.
	\end{cases}
\end{align*}
where $\bm{K}\subset \mathbb{R}^n$ is a Borel-measurable set of interest. The empirical measure is
\begin{align*}
	\mu_{s} := \frac{1}{s}\sum_{i=1}^s \delta_{\mathbf{x}^i},
\end{align*}
and the Monte Carlo estimator of the true probability $\mu(\bm{K})=\int \mathbf{1}_{\bm{K}}(\mathbf{x})\mathrm{d}\mu(\mathbf{x})$ is
\begin{align*}
	\hat{\mu}_{s}(\bm{K}):= \int \mathbf{1}_{\bm{K}}(\mathbf{x}) \, \mathrm{d}\mu_{s}(\mathbf{x}) = \frac{1}{s} \sum_{i=1}^s \mathbf{1}_{\bm{K}}(\mathbf{x}^i).
\end{align*}

\clearpage

By the strong law of large numbers $\hat{\mu}_s(\bm{K}) \overset{\mathrm{a.s.}}{\underset{}{\longrightarrow}} \mu(\bm{K})$ as $s\to \infty$. However, the convergence is almost sure, not uniform in $\bm{K}$, and provides no deterministic guarantee for finite $s$. In particular, for any finite $s$ and any $\epsilon>0$, there exists a set of realizations with positive (though possibly small) probability such that $|\hat{\mu}_s(\bm{K})-\mu(\bm{K})| > \epsilon$.

Consequently, the estimate $\hat{\mu}_s(\bm{K})$ is neither a certified upper nor lower bound on $\mu(\bm{K})$; it is merely an asymptotically consistent, but inherently stochastic, approximation. Moreover, by the Central Limit theorem, the normalized error satisfies
\begin{align*}
	\sqrt{s}(\hat{\mu}_s \left(\bm{K})-\mu(\bm{K}) \right) \underset{s\to \infty}{\longrightarrow} \mathcal{N}(0,\sigma^2)
\end{align*}
with asymptotic variance $\sigma^2:=\int (\mathbf{1}_{\bm{K}}(\mathbf{x})-\mu(\bm{K}))^2 \,  \mathrm{d}\mu(\mathbf{x})$. 
Thus, one may construct asymptotic confidence intervals, but these are probabilistic and not rigorous enclosures.

When applied to the problem setting of Sect.~\ref{sec:problem_setup}, the direct Monte Carlo approach suffers from three structural limitations:
\begin{enumerate}
	\item The method requires a fixed initial probability measure $\mu_0$ on the state space, satisfying prescribed moment constraints. However, the worst-case probability over the set of admissible initial measures is not attained by the chosen $\mu_0$. 
	\item The sample-based estimate $\hat{\mu}_s(\bm{K})$ is not a guaranteed upper bound; only probabilistic statements (e.g., “with high probability”) are available.
	\item Since the problem of interest involves a dynamical system over a continuous time horizon $\bm{T}=[0,T]$, a naive strategy would be to discretize time into a grid $0=t_0 < t_1 < \cdots < t_m = T$, propagate each sample $\mathbf{x}^i(t_k)$ through the dynamics, and evaluate $\mathbf{1}_{\bm{K}}$ at each time step. This introduces both temporal discretization error and a curse of dimensionality, as the empirical measure at each time must be estimated separately.
\end{enumerate}

Thus, while Monte Carlo simulation provides an efficient and asymptotically consistent estimator, it lacks the deterministic certification required for rigorous verification in safety-critical dynamical systems.

\section{Rational dynamical systems}%
\label{sec:appendix_C}

The theoretical framework developed in Sect.~\ref{sec:infinite_dimensional_program}-\ref{sec:finite_dimensional} rests on the fundamental assumption that the system dynamics are polynomial in the state and time variables, as stated in Assumption~\ref{ass:poly_dynamics}. While this assumption is satisfied for a broad class of systems, many applications of practical interest involve rational vector fields. In this section, we extend the occupation measure-based approach to accommodate dynamical systems governed by rational differential equations of the form
\begin{align}
	\label{eq:rational_dynamics}
	\dot{\mathbf{x}}(t) = f(\mathbf{x}, t) = \frac{p(\mathbf{x}, t)}{q(\mathbf{x}, t)}, \qquad (\mathbf{x}, t) \in \bm{X} \times \bm{T},
\end{align}
where $p\in\mathbb{R}[\mathbf{x}, t]^n$ and $q$ is an algebraic function in $(\mathbf{x}, t)$. Here we follow the same rationale from some previous works, where, by adding lifting variables and new equality constraints, a moment-SOS program is obtained, circumventing the rational dynamics.

To embed~\eqref{eq:rational_dynamics} into the polynomial framework, we introduce an auxiliary variable $u \in \bm{U} \subset \mathbb{R}$ and seek a polynomial function $\phi\in \mathbb{R}[\mathbf{x}, u, t]$ satisfying the equivalence
\begin{align}
	\dot{\mathbf{x}}(t) = \frac{p(\mathbf{x}, t)}{q(\mathbf{x}, t)}
	\quad \Longleftrightarrow \quad
	\dot{\mathbf{x}}(t) = \zeta(\mathbf{x}, u, t) \quad \text{s.t.} \quad \phi(\mathbf{x}, u, t) = 0,
\end{align}
for some suitable $\zeta \in \mathbb{R}[\mathbf{x},u,t]^n$.
Under this construction, the rational dynamics are exactly recovered on the algebraic manifold
\begin{equation*}
	\bm{\Phi} = \left\{(\mathbf{x}, u, t) \in \bm{X} \times \bm{U} \times \bm{T} : \phi(\mathbf{x}, u, t) = 0 \right\}.
\end{equation*}

The resulting system is thus a polynomial system evolving on a manifold. To extend our occupation measure framework to this setting, we lift the occupation measure to the augmented space by defining
$
	\mu \in \mathcal{M}_+(\bm{X} \times \bm{T} \times \bm{U}),
$
where the additional coordinate captures the manifold restriction. The Liouville equation~\eqref{seq:liouv} in the original measure program~\eqref{eq:measure-program} must be reformulated to account for the augmented state space. Specifically, we replace constraint~\eqref{seq:liouv} with
\begin{align}
	\label{eq:liouv-rational}
	\mu_\mathrm{T} = \delta_0 \otimes \mu_0 + \mathcal{L}^\dagger_{\zeta} \pi_{\#}^{\mathbf{x},t} \mu,
\end{align}
applying the $(\mathbf{x},t)$-marginalization to the occupation measure $\mu$ (see Sect.~\ref{sec:measure_program}).
To ensure that the lifted measure is supported on the algebraic manifold $\bm{\Phi}$, we impose the additional linear constraint
	$\mu \in \mathcal{M}_+(\bm{\Phi})$.
The complete augmented measure program for the rational dynamics then takes the form
\begin{subequations}
	\label{eq:measure_rational}
	\begin{align}
	\sup_{\mu, \mu_0, \mu_\mathrm{T}, \mu_\mathrm{a}} \quad & \langle 1, \mu_\mathrm{a} \rangle  \\
	\text{s.t.} \quad \quad \, & \mu_\mathrm{T} = \delta_0 \otimes \mu_0 + \mathcal{L}^\dagger_{\zeta} \pi_{\#}^{\mathbf{x},t} \mu, \\
	& \mu \in \mathcal{M}_+(\bm{X} \times \bm{T} \times \bm{U} \cap \bm{\Phi}), \label{seq:support_restriction_Phi} \\ 
	&\mu_\mathrm{T} \in \mathcal{M}_+(\bm{X} \times \bm{T}), \\
	& \eqref{seq:probability_1}\textendash\eqref{seq:mu_a_K}. 
	\end{align}
\end{subequations}

\begin{assumption}%
	\label{ass:strictly_sign_on_denominator}
	The denominator of the dynamics~\eqref{eq:rational_dynamics}, i.e., $q(\mathbf{x},t)$, is strictly negative or strictly positive on $\bm{X}\times \bm{T}$. 
\end{assumption}

\begin{theorem}
	For systems governed by the rational dynamics~\eqref{eq:rational_dynamics}, Program~\eqref{eq:measure_rational} yields a valid upper bound on the optimal value of Problem~\ref{prob:main}, provided that Assumptions~\ref{ass:compactness} and~\ref{ass:strictly_sign_on_denominator} hold.
\end{theorem}
\begin{proof}
	The claim follows by a direct application of the same reasoning employed in the proof of Theorem~\ref{theorem:measure_program}. Hence, the details are omitted for brevity.
\end{proof}

Similar to the procedure taken in Sect.~\ref{sec:moment_program}, we relax~\eqref{eq:measure_rational}, creating a sequence of finite-dimensional problems, i.e., a Lasserre hierarchy, making use of the moment and localizing matrix presented in Sect.~\ref{sec:moment_program}. Note that the new support restriction~\eqref{seq:support_restriction_Phi} consists of a manifold, hence in terms of moments constraints, this translates into the linear constraints
\begin{equation}
	\forall (\boldsymbol{\alpha}, \beta, \gamma) \in \mathbb{N}^{n+2}_{2d - \deg(\phi)}: \quad L_{\mathbf{y}}(\phi(\mathbf{x}, u, t) \, \mathbf{x}^{\boldsymbol{\alpha}} u^{\beta} t^{\gamma}) = 0,
\end{equation}
for the relaxation order $d$.

The moment program for each relaxation order $d$ reads
\begin{subequations}
	\label{eq:moment-program-rational}
	\begin{align}
		P^\star_d=\sup_{\mathbf{y}, \mathbf{y}_0, \mathbf{y}_\mathrm{T}, \mathbf{y}_\mathrm{a}} \quad & L_{\mathbf{y}_\mathrm{a}}(1) \\
		\mathrm{s.t.} \qquad &  \forall (\bm{\alpha},\beta) \in \mathbb{N}_{\leq 2d}^{n+1}: \quad \operatorname{Liouv}^u_{\bm{\alpha}, \beta}(\mathbf{y}_0, \mathbf{y}, \mathbf{y}_\mathrm{T}) = 0  \\
		& \forall (\boldsymbol{\alpha}, \beta, \gamma) \in \mathbb{N}^{n+2}_{2d - \deg(\phi)}: \quad L_{\mathbf{y}}(\phi(\mathbf{x}, u, t) \, \mathbf{x}^{\boldsymbol{\alpha}} u^{\beta} t^{\gamma}) = 0 \\
		& \eqref{seq:initial_moments}\textendash\eqref{seq:constraint_mu0_moments}
	\end{align}
\end{subequations}
where $\operatorname{Liouv}^u_{\bm{\alpha}, \beta}(\cdot)$ encodes the augmented Liouville equation~\eqref{eq:liouv-rational}, in terms of the pseudo-moments $(\mathbf{y}_0, \mathbf
{y}, \mathbf{y}_\mathrm{T})$, i.e.,
\begin{align}
	\langle \mathbf{x}^{\bm{\alpha}} t^{\beta}, \mu_\mathrm{T} \rangle -\langle \mathbf{x}^{\bm{\alpha}} t^{\beta}, \delta_0 \otimes \mu_0\rangle - \langle \mathcal{L}_\zeta(\mathbf{x}^{\bm{\alpha}} t^{\beta}), \pi_{\#}^{\mathbf{x},t}\mu\rangle = 0.
\end{align}

The introduction of the auxiliary variable $u$ increases the dimension of the state from $n$ to $n+1$, which inevitably increases the computational complexity of the resulting moment-SOS hierarchy; this hinders the application of this approach to the object in-orbit example, which suits this technique, since the dynamics are rational. 

In summary, the occupation measure framework is readily extensible to rational dynamical systems via an algebraic lifting to a higher-dimensional space, at the cost of introducing additional variables and linear equality constraints encoding the manifold structure. This extension preserves the convexity and the asymptotic convergence guarantees of the original hierarchy, while significantly broadening the class of systems amenable to analysis.

\end{appendices}


\bibliography{sn-bibliography}

\end{document}